\documentclass[12pt,a4paper]{article}
\usepackage{amsmath,amssymb,amsthm,latexsym,ifthen}
\usepackage[english]{babel}
\usepackage{graphicx}
\usepackage{comment}

\usepackage{anyfontsize}
\usepackage{multirow}
\usepackage[colorlinks, bookmarks]{hyperref}
\newtheorem{theorem}{Theorem}[section]
\newtheorem{lemma}[theorem]{Lemma}
\newtheorem{cor}[theorem]{Corollary}

\theoremstyle{definition}

\newtheorem{rmk}[theorem]{Remark}
\newtheorem{defin}[theorem]{Definition}
\newtheorem{notation}[theorem]{Notation}
\newtheorem{ex}[theorem]{Example}
\usepackage[english]{babel}

\begin{document}
 
\title{Integrability of
diagonalizable matrices and a dual Schoenberg type
inequality
}
\author{S.V. Danielyan, A.E. Guterman, T.W. Ng}
 
\date{\small $^1$Lomonosov Moscow State University, Moscow,   Russia\\$^2$Moscow Center for Fundamental and Applied Mathematics,
		Moscow,   Russia\\$^3$Moscow Institute of Physics and Technology,
		Dolgoprudny, Russia\\$^4$Department of Mathematics,
		The University of Hong Kong,
		Pokfulam, Hong Kong\\ dsvnja@gmail.com \ \ \ guterman@list.ru \ \ \ ntw@maths.hku.hk}
\maketitle

\begin{abstract}
The concepts of differentiation and integration for matrices were
introduced for studying zeros and critical points of complex polynomials.
Any matrix is differentiable, however not all matrices are integrable.
The purpose of this paper is to investigate the integrability property and characterize it within the class of
diagonalizable matrices. In order to do this we study the relation between the
spectrum of a diagonalizable matrix and
 its integrability and the diagonalizability of the integral.
Finally, we apply our results to obtain a dual Schoenberg type inequality
relating zeros of
polynomials with their critical points.


[MSC 2020] Primary: 30C10, 30C20, Secondary: 15A15

\end{abstract}

\def\char{{\rm char\,}}
\def\C{{\mathbb C}}
\def\K{{\mathbb K}}
\section{Introduction}

Zeros and critical points of a given univariable polynomial are related by
the celebrated Gauss-Lucas theorem stating that critical points belong to
the convex hull of zeros. More detailed and thorough investigation of the
relation between zeros and critical points of a polynomial is an interesting
open problem which stimulated already a lot of research in analytic theory of
polynomials. Far reaching progress in this area was established by the
solution of the famous Schoenberg conjecture proved independently by Pereira
\cite{Differentiators} and Malamud \cite{Malamud1,Malamud2}. This progress
was based on the notion of matrix differentiators introduced much  earlier
by Davis~\cite{Davis}.  In this paper we use the techniques from matrix
analysis and linear algebra to
study the inverse concept of integrator of a matrix which was
introduced by  Bhat and 
Mukherjee~\cite{IntegratorsOfMatricies}.
Bhat and Mukherjee have shown that any matrix is either freely integrable, or
uniquely integrable or non-integrable and characterized the freely
integrable matrices. However, the problem  to characterize the latter two categories was left open.
Our paper solves this problem for diagonalizable matrices in terms of characteristic polynomials.
In order to do this we study the relation between the spectrum of
diagonalizable matrix and
 its integrability and the diagonalizability of the
corresponding integral.
Moreover,
we apply the obtained results  to derive a dual Schoenberg type inequality
providing an upper bound for the sum of  squares of
the absolute values of zeros by an expression in
the critical points. 

We start by introducing some notations that we need further for  the main
definitions of our paper. In this paper denote by $\C$ the field of complex
numbers, and by $\K$ an
arbitrary algebraically closed field. If it is not  specifically mentioned, we
assume that the characteristic $\char \K=0$. Let $M_{n,m}(\K)$ denote the space
of all $n \times m$ matrices with entries from $\K$, we write $M_n$ if $m=n$.
Let $I_n$ be the unit matrix $n \times n$, $O_n$ be the zero $n \times n$
matrix. We write $I$ and $O$ if the size of the matrix is clear from the
context.
The transpose of   $A\in M_{n,m}$ is denoted by $A^{\top}\in M_{m,n}$.
Vectors in $\K^n$ are considered as row vectors and are identified with
corresponding $n$-tuples. The $j$'th unit vector is denoted by 	$e_j$ and
$e=(1, \ldots, 1)^{\top}$. For $X\in M_{n,m}(\K), Y\in M_{k,l}(\K)$ we denote by
$X\oplus Y\in M_{n+k,m+l}(\K)$ the block matrix
$\left(\begin{smallmatrix} X & 0 \\ 0 & Y\end{smallmatrix}\right).$
In case $\K = \C$ by $||A||_F$ we denote the
Frobenius norm of a matrix $A=(a_{ij})$,
i.e. $||A||_F = \sqrt{\sum\limits_{i, j} |a_{ij}|^2}$, by $||v||$
we denote the Euclidean norm of $n$ dimensional vector $v$, i.e.
$||v|| := \sqrt{\sum\limits_{i = 1}^{n} |v_i|^2}.$

The following notion of matrix   differentiability was
introduced by Davis in \cite{Davis} and further  investigated
in~\cite{IntegratorsOfMatricies,CheN06,CheN10,KushT16,Differentiators}.
\begin{defin}[{\cite[Definition~1]{IntegratorsOfMatricies}}]
\label{DEF:differentiator}
Let  $A$ be a linear operator  on a complex Hilbert space $H$ of the
dimension $n$ and  $P$ 
be an operator of an orthogonal projection on $H$ with
$\dim~\text{Ker}(P) = 1$. Let  $B$ be an operator satisfying the condition  
$B = PAP|_{P(H)}$. Then  $P$ is called a {\it differentiator}
of the operator  $A$, if characteristic polynomials of  $A$ and  $B$ satisfy
the condition $$p_{B}(x) = \frac{1}{n}p_{A}'(x).$$
In this case the operator $B$ is usually called a differential of the
operator~$A$.
\end{defin}

\medskip
Now without loss of generality we can assume that  
\begin{equation}
P = 
\left[
\begin{array}{cccc}
I_{n} & 0 \\
0 & 0
\end{array}
\right], \ 
A = 
\left[
\begin{array}{cccc}
B & u^{\top} \\
v & \lambda
\end{array}        
\right]
\end{equation}
where  $u, v \in \C^{n}$.

Differentiators appear to be useful in studying the relation between
the zeros of the polynomial and its critical points.
In 2003,   Pereira ~\cite{Differentiators} and
Malamud ~\cite{Malamud1, Malamud2} independently proved the following
conjecture using the method of differentiators of finite dimensional
operators.
\medskip

\begin{theorem}[{\bf Schoenberg's conjecture (1986)
\cite[Conjecture 3.1]{Differentiators}}]\label{THM:schoenberg}
Let $p(z)$ be a degree $n$ complex polynomial with zeros $z_1,\ldots,z_n$ and
critical points $w_1,\ldots, w_{n-1}.$ Then
$$\sum\limits^{n-1}_{i=1} |w_i|^2 \le
\left|\frac{1}{n}\sum\limits_{i = 1}^{n} z_i\right|^2 + 
\frac{n-2}{n} \sum\limits^n_{i=1} |z_i|^2$$
with the equality holds if and only if all $z_i$ lie on a straight line.
\end{theorem}

Comparing the coefficients at  $x^{n-1}$, we obtain
$tr(B) = \frac{n}{n +1}tr(A)$,
which implies that  $\lambda = tr(A) - tr(B) =
\frac{n + 1}{n}tr(B) - tr(B) = \frac{tr(B)}{n} =: \tau(B)$\\
The converse operation of integration was introduced by
Bhat and    Mukherjee in~\cite{IntegratorsOfMatricies}.
\begin{defin}[{\cite[Definition~3]{IntegratorsOfMatricies}}]
\label{DEF:integrator}
 Let  $B \in M_{n}(\C), \ A \in M_{n+1}(\C)$,
then  $A$ is called  {\it an integral  of } $B$,  if 
\begin{equation}
A = 
\left[
\begin{array}{cccc}
B & u^{\top} \\
v & \tau(B)
\end{array}        
\right],
\end{equation} and also  $p_{B}(x) = \frac{1}{n+1}p_{A}'(x)$. 
In this case the pair of vectors  $(u, v)$ is called  {\it
an integrator} of  $B$ and the element
$\det(A)$  {\it is called a constant of integration}.
\end{defin}

For any algebraically closed field $\K$ with $\char \K = 0$ one can
define a formal derivative and
an integral of a polynomial  $p(x) = a_nx^n + \ldots + a_1x + a_0$  as
follows. The derivative is $p'(x) := na_nx^{n-1} + \ldots + 2a_2x + a_1$,
and the integral
$P(x)  := \frac{1}{n+1}a_nx^{n+1} + \ldots + \frac12 a_1x^2 + a_0x+C,$
where $C\in \K$ is a constant. Then Definitions \ref{DEF:differentiator} 
and \ref{DEF:integrator} can be considered for matrices over any
algebraically closed field $\K$ of characteristic zero.
It is straightforward to check that the results
from \cite{IntegratorsOfMatricies} are also true in this generality.
In this paper we investigate how the integrability depends on the
values of zeros, their multiplicities, and integrability
property. Thus below we consider algebraically closed fields of zero
characteristic.

\begin{defin}[{\cite[Definition~5]{IntegratorsOfMatricies}}] 
The square matrix  $B$ is called 
{\it integrable} if there exists its integral.  A matrix $B$ is 
called {\it uniquely integrable} if it is integrable and there
exists  $\alpha \in \C$ such that 
for any integral  $A$ of the matrix  $B$ it holds that $\det(A) = \alpha$.
A matrix $B$ is called  {\it freely integrable} if
for any  $\alpha \in \C$ there exists an integral $A$
of the matrix  $B$, such that  $\det(A) = \alpha$.
\end{defin}

Below we collect several examples demonstrating different properties
and features of matrix integrability.

Let us start with an example of freely integrable matrix.
\begin{ex} \cite[Example 4]{IntegratorsOfMatricies} \label{EX:freely_integrable}
Fix  $\lambda \in \K\setminus \{1\}$. Consider
$B = 
\begin{pmatrix}
1 & 0 \\
0 & \lambda
\end{pmatrix}$. 
Observe that for any $ t \in \K$, \\ 
$A_t = 
\begin{pmatrix}
1 & 0 & 1 \\
0 & \lambda & 1 \\
\frac{2t - 3\lambda + 1}{2(1 - \lambda)} &
-\frac{(\lambda - 1)^{3} + 2t - 3\lambda + 1}{2(1 - \lambda)} &
\frac{\lambda + 1}{2}
\end{pmatrix}$ is an integral of B with the constant of integration~$t$.

Indeed,
$\det(xI - A_t) = x^3 - \frac{3(\lambda + 1)}{2}x^2 + 3\lambda x - t$,
therefore\\
$p_{A_t}'(x) = 3x^2 - 3(\lambda+1)x + 3\lambda = 3(x -1)(x - \lambda) =
3p_{B}(x)$
and $\det(A_t) = t.$
Therefore $B$ is freely integrable.
\end{ex}
We now give an example of uniquely integrable matrix.
\begin{ex} \label{EX:uniquely_integrable}
Consider the case
$\lambda = 1,$ i.e. 
$B = 
\begin{pmatrix}
1 & 0 \\
0 & 1
\end{pmatrix}$ and write the integral in its general form
$A =
\begin{pmatrix}
1 & 0 & u_1 \\
0 & 1 & u_2 \\
v_1 & v_2 & 1
\end{pmatrix}$. We have that
$$p_{A}(x) = x^3 - 3x^2 + (3 - u_1v_1 - u_2v_2)x + (u_1v_1 + u_2v_2 - 1),$$
$$p_{A}'(x) = 3x^2 - 6x + (3 - u_1v_1 - u_2v_2) = 3p_{B}(x) = 3x^2 - 6x + 3,$$
$$3 - u_1v_1 - u_2v_2 = 3.$$
Last equation has solutions and 
$\det(A) = 1 - (u_1v_1 + u_2v_2) = 1$ 
for any solution. This implies that  $B$ is uniquely integrable.
\end{ex}
Finally consider an example of non-integrable matrix.
\begin{ex}\cite[Example 4]{IntegratorsOfMatricies} \label{EX:not_integrable}
In \cite[Theorem 16]{IntegratorsOfMatricies} it was shown that any  
matrix of size two or three is integrable. However if $\lambda_1 \neq \lambda_2$
then, for example, no integral exists for the following matrix of the size four:
$B_{\lambda_1,\lambda_2} := {\rm diag\,}
(  \lambda_1, \lambda_1 , \lambda_2 ,   \lambda_2),$ i.e., the diagonal matrix
with the entries $\lambda_1, \lambda_1 , \lambda_2 ,   \lambda_2$ on the
diagonal.

Indeed, assume that 
$A = \left(\begin{smallmatrix}
    \lambda_1   && 0   && 0   && 0   && u_1\\
    0   && \lambda_1   && 0   && 0   && u_2\\ 
    0   && 0   && \lambda_2   && 0   && u_3\\
    0   && 0   && 0   && \lambda_2   && u_4\\
    v_1 && v_2 && v_3 && v_4 && \tau(B_{\lambda_1,\lambda_2})
\end{smallmatrix}\right)$
is an integral of $B_{\lambda_1,\lambda_2}$.\\
Then $p'_A(x) = 5p_{B_{\lambda_1,\lambda_2}}(x)$.
Writing down the determinant of $xI - A$ we obtain
$$p_A(x) = (x - \tau(B_{\lambda_1,\lambda_2}))p_{B_{\lambda_1,\lambda_2}}(x) -
(u_1v_1 + u_2v_2)(x - \lambda_1)(x - \lambda_2)^2 -
(u_3v_3 + u_4v_4)(x - \lambda_1)^2(x - \lambda_2).$$
By direct substitution we have $p_A(\lambda_1) = p_A(\lambda_2) = 0$.
Thus the multiplicities of $\lambda_1, \lambda_2$ as zeros of
$p_A(x)$ are equal to $3$,
which is not possible since $\deg(p_A) = 5.$
\end{ex}

The following example shows that integration can break many matrix properties
such as   unitarity.
 
\begin{ex}
Let $\K = \C.$ Consider an arbitrary unitary non-scalar operator
$B$ and its integral $A.$
It turns out that $A$ is not unitary. Indeed, suppose that $A$ is unitary.
Denote $n = \deg(p_{A}(x)).$
Since $p_{B}(x) = \frac{p_{A}'(x)}{n}$ then by  Theorem \ref{THM:schoenberg} we obtain
$$\sum\limits^{n-1}_{i=1} |w_i|^2 \le
\left|\frac{1}{n}\sum\limits_{i = 1}^{n} z_i\right|^2 + 
\frac{n-2}{n} \sum\limits^n_{i=1} |z_i|^2,$$
where $z_1,\ldots, z_n$ are the zeros of $p_{A}(x)$, and  $w_1,\ldots,w_{n-1}$ are
the zeros of $p_{B}(x).$ Since the spectrum of unitary operator lies on the
unit sphere we conclude
$$|z_1| = \ldots = |z_n| = |w_1| = \ldots = |w_{n-1}| = 1.$$
Therefore
$$n - 1 \leq \left|\frac{1}{n}\sum\limits_{i = 1}^{n} z_i\right|^2 +
n\cdot \frac{n-2}{n} \text{\ \ \ and \ \ \ } \left|\sum\limits_{i = 1}^{n} z_i\right|^2 \geq n.$$
Thus $z_1 = \ldots = z_n,$ hence $w_1 = \ldots = w_{n-1}$, which contradicts
with $B$ being non-scalar.
\end{ex}
\begin{rmk}
The result of the previous example could be obtained by only using
Gauss-Lucas theorem. Suppose that $A$ is unitary.\\
1) If $A$ is scalar then $B = PAP,$ where $P$ is an operator of
orthogonal projection, is also scalar.\\
2) If $A$ is not scalar then, since any unitary operator is diagonalizable,
it has at least two distinct eigenvalues. Then
$p_{B}(x) = \frac{p_{A}'(x)}{\deg(p_{A}(x))}$
has a zero $x_0,$ such that $p_{A}(x_0) \neq 0.$
By the Gauss-Lucas theorem $x_0$ lies in the convex hull of the zeros of
$p_{A}(x).$ Since the spectrum of unitary operator lies on the unit sphere we
obtain that $x_0$ lies in the open unit disc. Thus the spectrum of $B$
does not lie on the unit sphere, which contradicts with the
unitarity of $B.$
\end{rmk}

Next we consider an example that shows us how one can construct non-integrable
matrices of any even order $\geq 4.$

\begin{ex}
Consider a diagonal matrix $B_0 = {\rm diag\,}(\lambda_1, \ldots, \lambda_n)$
and $B = \left( \begin{smallmatrix}
B_0 & 0\\
0 & B_0
\end{smallmatrix}\right).$ Then $B$ is integrable if and only if
$B_0 = \lambda I_{n}.$
Indeed, if the integral exists, then it has the form
$A = \left( \begin{smallmatrix}
B_0 & 0 & u_1\\
0 & B_0 & u_2\\
v_1 & v_2 & \tau(B)
\end{smallmatrix}\right).$
Denoting $X = xI_{2n} - B, \ Y = -(u_1, u_2)^{\top}, \ 
Z = -(v_1, v_2), \ W = xI_1 - \tau(B), $
by the formula for the determinant of block matrix
$$
\det\left(\begin{smallmatrix}
X & Y\\
Z & W
\end{smallmatrix}\right) = \det(X)\det(W - ZX^{-1}Y)
$$
we obtain that for any $x$ distinct from the zeros of $p_{B_0}(x)$
$$p_A(x) = \det(xI_{2n+1} - A) = 
p_{B_0}^2(x)(x - \tau(B) - (v_1, v_2)(xI_{2n} - B)^{-1}(u_1, u_2)^{\top}).
$$
Since
$(xI_{2n} - B)^{-1} =
{\rm diag\,}(\frac{1}{x - \lambda_1}, \ldots, \frac{1}{x - \lambda_n},
\frac{1}{x - \lambda_1}, \ldots, \frac{1}{x - \lambda_n})$ then 
$$(v_1, v_2)(xI_{2n} - B)^{-1}(u_1, u_2)^{\top} = \sum\limits_{i = 1}^{n}
\frac{t_i}{x - \lambda_i} \text{ for some } t_1, \ldots, t_n \in \K.$$ 
Thus $p_{B_0}(x)\biggl((v_1, v_2)(xI_{2n} - B)^{-1}(u_1, u_2)^{\top}\biggl)$
is a polynomial, which leads to $p_{B_0}(x)\ | \ p_{A}(x).$
Therefore if $A$ is an integral of $B$ then
$$p_{A}'(x) = (2n+1)p_{B_0}^2(x).$$ Thus any zero of
$p_{A}'(x)$ is a zero of $p_{A}(x).$ Therefore
$p_{A}(x) = (x - \lambda)^{2n+1}, \ p_{B}(x) =
(x - \lambda)^{2n}, \ B = \lambda I_{2n}.$
\end{ex}

In the next example we show how an integrable non-scalar matrix
with multiple eigenvalues can be produced.

\begin{ex}
Consider the matrix $B = \left( \begin{smallmatrix}
aI_{n-1} & 0\\
0 & b
\end{smallmatrix}\right)$ with $a \neq b.$
Then the matrix $A = \left( \begin{smallmatrix}
aI_{n- 1} & 0&0\\
0& b & 1\\
0& \beta & \tau(B) 
\end{smallmatrix}\right),$ where $\beta = b\tau(B) - a(\tau(B) + b - a),$
is an integral of $B.$ Indeed,
$$
p_A(x) = (x - a)^{n-1}((x - b)(x - \tau(B)) - b\tau(B) + a(\tau(B) + b - a))=$$
$$  = (x - a)^{n-1}(x^2 - (b + \tau(B))x + a(\tau(B) + b - a)) =
(x - a)^{n}(x - (\tau(B) + b - a)).$$
Then 
$
p_A'(x) = n(x - a)^{n-1}(x - (\tau(B) + b - a)) + (x - a)^{n}  =
(x - a)^{n-1}(nx - n(\tau(B) + b - a) + x - a) =
(x - a)^{n-1}((n+1)x -
n\left(\frac{(n-1)a + b}{n} + b - a\right) - a)  = 
(n+1)(x - a)^{n-1}(x - b) = (n+1)p_{B}(x).
$
\end{ex}

It turns out that an integral of a diagonal matrix could be both
diagonalizable or not, as the following example shows.

\begin{ex}
Consider $B = I_n,$ then $A_1 = I_{n+1}$ is a diagonalizable integral of
$B,$ but it is straightforward to see that $A_2 = \left( \begin{smallmatrix}
I_{n} & 1\\
0 & 1
\end{smallmatrix}\right)$ is also an integral of $B,$ since
$p_{A_2}(x) = p_{A_1}(x),$
but is not diagonalizable.
\end{ex}

In previous examples the integrability does not depend on the particular values of
eigenvalues, it only depends on their multiplicities. This in not the case in
general, as the following example shows.

\begin{ex}
Consider $B = {\rm diag}\, (a,a,b,b,c).$ 
Let us show that if $A = \left(\begin{smallmatrix}
a&0&0&0&0&u_1\\
0&a&0&0&0&u_2\\
0&0&b&0&0&u_3\\
0&0&0&b&0&u_4\\
0&0&0&0&c&u_5\\
v_1&v_2&v_3&v_4&v_5&\tau(B)
\end{smallmatrix}\right)$
is an integral of $B$
then $p_{A}(x) = (x - a)^3(x - b)^3.$ Indeed,
$$
p_{A}(x) = (x - a)^2(x - b)^2(x - c)(x - \tau(B)) -
(v_1u_1 + v_2u_2)(x - a)(x - b)^2(x - c)- $$
$$-  (v_3u_3 + v_4u_4)(x - a)^2(x - b)(x - c) -
v_5u_5(x - a)^2(x - b)^2.
$$
Thus if $p_{A}'(x) = 6p_{B}(x)$ then
$$p_{A}(a) = p'_{A}(a) = p''_{A}(a) = p_{A}(b) =
p'_{A}(b) = p''_{A}(b) = 0,$$
hence
$$
v_1u_1 + v_2u_2 = v_3u_3 + v_4u_4 = 0,
$$
$$
p_{A}(x) = (x - a)^2(x - b)^2(x - c)(x - \tau(B)) -
v_5u_5(x - a)^2(x - b)^2.
$$
Therefore
$$
p_A''(a) = 2(a - b)^2((a - c)(a - \tau(B)) - v_5u_5) = 0,$$
$$p_A''(b) =
2(a - b)^2((b - c)(b - \tau(B)) - v_5u_5) = 0,
$$
thus $(x - c)(x - \tau(B)) - v_5u_5 = (x - a)(x - b)$ and
$p_{A}(x) = (x - a)^3(x - b)^3.$

Thus
$p_{A}'(x) = 6(x - a)^2(x - b)^2(x - \frac{a + b}{2}).$ Hence if
$c \neq \frac{a + b}{2}$ then $B$ is not integrable. On the other hand
if $c = \frac{a + b}{2}$ then
$$A = \left(\begin{smallmatrix}
a&0&0&0&0&0\\
0&a&0&0&0&0\\
0&0&b&0&0&0\\
0&0&0&b&0&0\\
0&0&0&0&c&1\\
0&0&0&0&c\tau(B) - ab&\tau(B)
\end{smallmatrix}\right)$$
is an integral of $B$. Indeed,
$$
p_{A}(x) = (x - a)^2(x-b)^2((x - c)(x - \tau(B)) + ab - c\tau(B)),$$
$$p_A(x) = (x - a)^2(x-b)^2(x^2 - (c + \tau(B))x + ab),$$
$$p_A(x) = (x - a)^2(x-b)^2(x^2 - (a + b)x + ab) = (x - a)^3(x-b)^3,
$$
$$p_A'(x) = 6(x - a)^2(x - b)^2
\left(x - \frac{a + b}{2}\right) = 6p_{B}(x).$$
\end{ex}

In  \cite[Corollary 10]{IntegratorsOfMatricies} it was proved that
the following alternative
holds: a matrix  is either freely
integrable, or uniquely integrable, or non-integrable. It is shown in
\cite[Theorem 9]{IntegratorsOfMatricies} that a matrix is freely
integrable if and only if it is non-derogatory.
However, the recognition question for integrability or non-integrability of a given
matrix remained open.

In this paper we  investigate   integrability for any
diagonalizable matrices in terms of the multiplicities of their eigenvalues.
We find the conditions on multiplicities that determine if a matrix is integrable
or not and show that for all other tuples there are both integrable and
non-integrable matrices. Also we present a criterion for the diagonalizability
of an integral of a diagonalizable matrix. Multiple integration is considered as
well. As a corollary we characterize sequences of diagonalizable matrices in
which each term is an integral of the previous one.
 
Our paper is organized as follows.
Section 2 is devoted to studying the
existence of a full integral of a polynomial (Theorem \ref{THM:full_int_classification}).
Section 3 describes the relation
between the integrability of a diagonalizable matrix and full integrals of
its
characteristic polynomial (Theorem \ref{THM:integrable_iff_fullintegral}). This section also describes the integrability for any
diagonalizable matrices in terms of the multiplicities of their eigenvalues (Theorem \ref{THM:diagonal_integrable_classifications}). In
Section 4 we provide a criterion for an integral to be diagonalizable and
consider  multiple integration (Theorem \ref{THM:diagonal_integral_diagonalizable_criteria} and Corollary \ref{COR:seq_int_iff_seq_fint}). In Section 5, we apply our results to obtain a
dual Schoenberg type inequality for polynomials with full integrals (Theorem \ref{THM:dual Schoenberg} and Corollary \ref{COR:dual Schoenberg}).


\section{Existence of a full integral of a polynomial}

In this section we   investigate the relations between the zeros of a
characteristic polynomial of a matrix, and its integral. It turns out
that these
questions can be reduced to the results on the full integrals of polynomials
from  the recent
paper \cite{DanG}.

\begin{defin}
\cite[Definition 1.1]{DanG}
\label{DEF:km}
We say that $p$ is a polynomial  {\em of type\/} $(k, m)$, where $k$ and $m$ 
are non-negative integers,
if  $p$ has  $k$ different simple zeros and    $m$  different 
multiple zeros.
\end{defin}

Consider some polynomials of different  types.
\begin{ex} Assume that $\lambda_1 \neq \lambda_2$. Then
$(x - \lambda_1)(x - \lambda_2)$ is the polynomial of type $(2, 0)$,\\
$(x - \lambda_1)(x - \lambda_2)^7$ is the polynomial of type $(1, 1)$,\\
$(x - \lambda_1)^4(x - \lambda_2)^5$ is the polynomial of type $(0, 2)$.
\end{ex}

\begin{defin} \label{DEF:full_integral}
\cite[Definition 1.3]{DanG}
The polynomial  $P \in \K[x]$ is called a {\em full integral\/} of
the polynomial $p \in \K[x]$, if $P' = p$ and
for any  $\lambda \in \K$ satisfying  $(x - \lambda)^2 | p$
we have that $(x - \lambda) | P.$ 
In other words,  any  multiple zero of the polynomial  $p$ 
is  a zero of~$P$.
\end{defin}

\begin{rmk}
The existence of full integrals splits into cases similarly
as the existence of integrals of the matrices do.
As was shown in \cite[Lemmas 3.1, 3.2]{DanG}:

1) If $m = 0$ then any integral of $p$ is its full integral
(which is similar to freely integrable matrix);

2) If $m > 0$ then either $p$ does not have a full integral
(similar to non-integrable matrix) or $p$ has unique full integral
(similar to uniquely integrable matrix).
\end{rmk}

Consider several examples.
\begin{ex} Let $p(x) = x^2$.

1. Polynomial $P_1(x) = x^{3}$ is a full integral of the polynomial  $p(x)$,
since  $0$ is a zero of the polynomial  $P_1(x)$.

2. Polynomial $P_2(x) = x^{3} + 1$ is an  integral of the polynomial  $p(x)$,
but it is not a full integral since  $0$ is not a zero of the
polynomial~$P_2(x)$.
\end{ex}

The full integrals of polynomials are closely related with the integrals of
diagonalizable matrices. In the next section we will show that a 
diagonalizable matrix is integrable if and only if 
its characteristic polynomial has a full integral.
Let us introduce the following notations and use them further.

\begin{notation} \label{DEF:f_q_Q_h}
By polynomial $f$ we denote the polynomial of type $(k, m)$
$$f(x) = (x - a_1)\ldots(x - a_k)
(x - b_1)^{\alpha_1}\ldots(x - b_m)^{\alpha_m},$$
where $a_1,\ldots, a_k, b_1, \ldots, b_m\in \K$ are pair-wise distinct,
$k, m \in \mathbb N \cup \{0\}, \ \alpha_1,\ldots, \alpha_m \in
\mathbb N \setminus \{1\}.$
We also denote
$$q(x) := (x - b_1)^{\alpha_1}\ldots(x - b_m)^{\alpha_m},$$
$$Q(x) := q(x)(x - b_1)\ldots(x - b_m),$$
$$h(x) := (x - a_1)\ldots(x - a_k).$$

Let $\K_l[x]$ be the linear space of polynomials of the degree
less than or equal to  $l$.

We denote by $U_i\subseteq \K_l[x]$ the subspace of polynomials having
zero value in 
$b_i,\ i = 1,\ldots, m.$ $U_0\subset \K_l[x]$ 
denotes the subspace of polynomials of the degree strictly less than $l$, and
$U = U_0\,\cup\,U_1\,\cup\,\ldots\,\cup\,U_{m}$.

For a fixed polynomial $f$ we consider the map 
$$\varphi_{l, m}: \K_{l}[x] \longrightarrow \K_{l + m - 1}[x],$$
defined by
$$\varphi_{l, m}: g \longmapsto \frac{(Qg)'}{q}.$$
\end{notation}

Below we provide several results concerning full integrals of polynomials proved
in \cite{DanG} since they appeared to be useful for matrix integrability.

\begin{lemma} \cite[Lemmas 2.7, 2.8]{DanG} \label{LM:varphiproperties}
The map $\varphi_{l,m}$ has following properties:

1. $\varphi_{l,m}$ is a linear map; 

2. the kernel $\text{Ker}\,\varphi_{l,m} = 0;$

3. if $m > 1$, then
$(\text{Im}\,\varphi_{l,m} \cup U) \subset  \K_{l + m - 1}[x]$,
and this inclusion is strict;

4. if $m = 1,$ then $\varphi_{l,m}$ is invertible;

5. the image $\text{Im}\,\varphi_{l,m} \nsubseteq U.$
\end{lemma}

\begin{lemma} \label{LM:manymultipleroots_nofullintegral}
\cite[Theorem 3.7]{DanG}
Let $m > k + 1$. Then  $f$ does not have a full integral.
\end{lemma}

\begin{lemma} \label{LM:singlemyltipleroot_has_fullintegral}
\cite[Theorem 3.8]{DanG}
Let $m = 1$. Then $f$ has a full integral.
\end{lemma}

\begin{lemma} \cite[Lemma 2.18]{DanG} \label{LM:nomultiplerootsinsum}
Let $f, g \in \K[x]$ be polynomials without common multiple zeros.
Then the set 
$T := \{ t \in \K\,|\, f + tg\in 
\K[x]\ \text{has a multiple zero}\}$ is finite.
\end{lemma}

\begin{lemma} \label{LM:integratordoesnotexists_helper}
Let $f$ has a full integral. Then $f \in q\cdot \text{Im}\,\varphi_{k-m+1,m}$.
\end{lemma}
\begin{proof}
Let   $F\in \K[x]$ be a full integral of   $f$.
Since all  multiple zeros of  $f$ are $b_1,\ldots, b_k$,
and they are  zeros of $F$, it follows that 
$F = Qg,$ for some $g \in \K[x].$
Thus
$qh = f = F' = (Qg)',$
i.e. $h = \frac{(Qg)'}{q} = \varphi_{k-m+1,m}(g).$
Therefore  $h \in \text{Im}\,\varphi_{k-m+1,m},$ and thus
$f \in q\cdot \text{Im}\,\varphi_{k-m+1,m}$.
\end{proof}

\begin{lemma} \label{LM:generalcase_integratordoesnotexists}
Let $m > 1$. Then there exist   $a_i, b_j$ such
that $f$ does not possess a full integral.
\end{lemma}

\begin{proof}
If  $k + 1 < m,$ then by Lemma \ref{LM:manymultipleroots_nofullintegral}
any polynomial $f$ does not possess a full integral.

Thus further we can assume that  $k + 1 \geq m.$

Denote $\varphi := \varphi_{k - m + 1, m}$ and consider $b_1,\ldots, b_m$
being fixed.

By  Item 3 of   Lemma \ref{LM:varphiproperties}
we can find the polynomial 
$g \in \K_{k}[x] \setminus (\text{Im}\,\varphi \cup U).$
Now consider the family of polynomials  $H := \{ g + c\,|\,c \in \K\}$. 
If two different polynomials $g + c_1, \  g + c_2 \in  \text{Im}\,\varphi$ then 
$$c_2(g + c_1) - c_1(g + c_2) = (c_2 - c_1)g \in \text{Im}\,\varphi,$$
which contradicts to the choice of  $g$.
Hence, $|\text{Im}\,\varphi \cap H|\le 1$.

Similarly, for any $i=0,\ldots,m$ if two different polynomials
$g + c_1, \  g + c_2 \in  U_i$, then 
$$c_2(g + c_1) - c_1(g + c_2) = (c_2 - c_1)g \in U_i,$$

thus $|U_i \cap H|\le 1$ for all $i=0,\ldots,m$.

Therefore, we obtain that the set  
$H_0 := H \cap (\text{Im}\,\varphi \cup U)$ 
is finite.

Moreover from Lemma \ref{LM:nomultiplerootsinsum} we obtain that the set
$H_1 := \{ r \in H\,|\, r\ \text{has multiple zero} \}$
is also finite.

Since the set $H$ is infinite then the set
$H \setminus (H_0 \cup H_1)$ is non-empty.
Choose an arbitrary polynomial $h \in H \setminus (H_0 
\cup H_1)$ and observe that the polynomial $f = qh$
satisfies the conditions of the lemma. Indeed, 

1) The polynomial $f$ has the form
$f = (x - a_1)\ldots(x - a_k)(x - b_1)^{\alpha_1}\ldots(x - b_m)^{\alpha_m}$
since  $h(b_i) \neq 0,\ i = 1,\ldots,m$ and $h$ has no multiple zeros
and  $\deg(h) = k$.

2) Since  $h \notin \text{Im}\, \varphi$ then
$f \notin q\cdot\text{Im}\,\varphi$. 

Therefore from the Lemma \ref{LM:integratordoesnotexists_helper} 
we obtain that the polynomial  $f$ does not possess a full integral.  
\end{proof}

\begin{lemma} \label{LM:generalcase_itegratorexists}
Let $m > 1$ and $k +1 \geq m$. Then there exist  $a_i, b_j$ such that $f$
possesses a full integral.
\end{lemma}
\begin{proof}
Consider pair-wise different
$b_1,\ldots,b_m \in \K$  such that  the polynomial
$Q(x) := (x - b_1)^{\alpha_1+1} \: \ldots \: (x - b_m)^{\alpha_m+1}$
is a full integral of its derivative. Such $b_1,\ldots,b_m$ exist by~\cite[Theorem 3.3]{DanG}.

Consider the map  $\varphi := \varphi_{k - m + 1, m}$  from
Definition \ref{DEF:f_q_Q_h}. By Lemma \ref{LM:varphiproperties}, Item 5,
there exists a polynomial
$h_1 \in \text{Im}\,\varphi\setminus(\text{Im}\,\varphi \cap U).$

Case 1. Let $k + 1 = m$.
Set $h_2 := h_1$. By its definition $h_1 = c\frac{Q'}{q}$ for some 
$ c \in \K$. Then the  polynomials  $h_1, h_2$
do not have multiple zeros.
Indeed,   $Q$ is a full integral of  $Q'$ and therefore any multiple zero of
$Q'$ is a zero of  $Q$. Then any multiple zero of $Q'$ is equal to  $b_i$ for
some $i=1,\ldots,m$.
However,  $\frac{Q'}{q}(b_i) \neq 0$ for any $i=1,\ldots,m$. Thus 
$\frac{Q'}{q}$, $h_1$, and
$h_2$ do not have  multiple zeros.

Case 2. Let $k + 1 > m$.
Denote by $x_1,\ldots,x_k$ the zeros of the  polynomial $h_1$.
From the definition of   $h_1$ we have that  $x_i \neq b_j$ for all
$i=1,\ldots, k$, $j=1,\ldots, m$. Therefore,  $Q(x_i) \neq 0, \ q(x_i) \neq 0, \ i = 1,\ldots,k.$

Denote  $W_i=\{r(x)\in \K[x]\vert r(x_i)=0\},\ i = 1,\ldots,k$.
Let us show that
$\text{Im}\,\varphi \nsubseteq W := W_1\,\cup\,\ldots\,\cup\,W_k\,\cup U$.
Indeed, we consider  $\varphi(x + c)$, where $c \in \K$. 
$$\varphi(x + c) = \frac{((x + c)Q)'}{q} = \frac{Q + (x + c)Q'}{q}.$$
If $Q'(x_i) \neq 0,$ then for
$c = -\frac{Q(x_i) + x_iQ'(x_i) - q(x_i)}{Q'(x_i)}$, 
we obtain that  $\varphi(x + c)(x_i) = 1 \neq 0$.\\
If  $Q'(x_i) = 0,$ then since  $Q(x_i) \neq 0$ we have that
$\varphi(x + c)(x_i) = \frac{Q(x_i)}{q(x_i)} \neq 0$.

Therefore $\text{Im}\,\varphi \cap W_i,\ i =1,\ldots,k,$ are proper
subspaces of $\text{Im}\,\varphi$. Since $\text{Im}\,\varphi \not\subseteq U,$
then $\text{Im}\,\varphi \cap U$ is a proper
subspace of $\text{Im}\,\varphi.$ Thus due to
\cite[Theorem 1.2]{ProperSubspaces}  
$\text{Im}\,\varphi \not\subseteq W_1\,\cup\,\ldots\,\cup\,W_k\,\cup U$.
Hence $\text{Im}\,\varphi \nsubseteq W.$

Let us consider the polynomial
$h_2 \in \text{Im}\,\varphi\setminus(\text{Im}\,\varphi \cap W)$.
Since $h_2 \notin W,$ the polynomials  $h_1$ and $h_2$ have no common zeros.
In particular, they do not have common multiple zeros.

So, in both cases above we constructed two polynomials 
$h_1, h_2 \in \text{Im}\,\varphi \setminus U$ that do not have
common multiple zeros.
Denote by
$$H := \{h_1 + th_2\,|\,t \in \K \} \subset \text{Im}\,\varphi,$$
$$H_0 := \{h_1 + th_2\,|\,t \in \K,\ h_1 + th_2 \in
\text{Im}\,\varphi \cap U\} \subset H,$$
$$H_1 := \{h_1 + th_2\,|\,t \in \K,\ h_0 + th_1\ 
\text{has multiple zeros} \}\subset H$$
From Lemma \ref{LM:nomultiplerootsinsum} we have that the set 
$H_1$ is finite.

Let us show that the cardinality    $|H_0|$ is at most~$1.$

Assume that  $h_1 + t_1h_2, \ h_1 + t_2h_2 \in H_0,$ $ t_1 \neq t_2$.
Then $h_1 + t_1h_2, \ h_1 + t_2h_2 \in U$. Therefore,
$(h_1 + t_1h_2) - (h_1 + t_2h_2) = (t_1 - t_2)h_2 \in U$.
Thus $h_2 \in U$ which contradicts to the definition of~$h_2$.

Therefore since the set  $H \subset \text{Im}\,\varphi$ is infinite and
the sets $H_0$ and $H_1$
are finite we can choose a certain polynomial
$h \in H \setminus(H_0 \cup H_1)$.

It remains to show that the polynomial  $f := qh$ 
satisfies the conditions of the lemma. Indeed,

1)   $f$ has the form  $f =
(x - a_1)\ldots(x - a_k)(x - b_1)^{\alpha_1}\ldots(x -b_m)^{\alpha_m}$,
since  $h(b_i) \neq 0$, by its construction $h$ does not
have multiple zeros, and we have  $deg(h) = k$.

2) The polynomial  $Q\varphi^{-1}(h)$ is a full integral  of the polynomial
$f$ since an arbitrary multiple zero of   $f$ is a zero of  $q$.
Therefore,  the same holds for the zeros of polynomial $Q$ and, moreover, for
the polynomial $Q\varphi^{-1}(h)$. Also   
$$f = qh  = q\varphi(\varphi^{-1}(h)) = q\frac{(Q\varphi^{-1}(h))'}{q} =
(Q\varphi^{-1}(h))'.$$
\end{proof}

\begin{theorem} \label{THM:full_int_classification}
Let $f \in \K[x]$ be a polynomial of the type
$(k, m), \ k\ge 0, m\ge 0.$
Then the following alternative is true:

1) If $m \leq 1$ then the polynomial  $f$ has a full integral.

2) If  $m > k + 1$ then the polynomial $f$ does not have a full integral.

3) For any pair $(k,m)$ which does not satisfy 1) and 2) and any sequence
$\alpha_1,\ldots,\alpha_m$  such that
$\alpha_i > 1,$ there are both possibilities:

a) there exists a polynomial $f_1$ of the type $(k,m)$ with the multiplicities
$\alpha_1,\ldots,\alpha_m$ of multiple zeros, such that there exists a full
integral of $f_1$, and

b) there exists a polynomial $f_2$ of the type $(k,m)$ with the multiplicities
$\alpha_1,\ldots,\alpha_m$ of multiple zeros, such that there is no
full integral of~$f_2$.
\end{theorem}
\begin{proof}
The first item is a direct application of
Lemma \ref{LM:singlemyltipleroot_has_fullintegral}.
The second one is proved in  Lemma
\ref{LM:manymultipleroots_nofullintegral}.
Lemma \ref{LM:generalcase_itegratorexists} implies Condition 3a) and 
Lemma \ref{LM:generalcase_integratordoesnotexists} implies Condition~3b).
\end{proof}

\section{Matrix integrability and full integrability of polynomials}
The following result is proved in \cite{IntegratorsOfMatricies} for
matrices over the field of complex numbers, however, its proof holds for an arbitrary field~$\K.$ 

\begin{lemma} \cite[Lemma~7]{IntegratorsOfMatricies}
\label{LM:integral_basis_change}
If $A \in M_{n+1}(\C)$ is an integral of $B \in M_{n}(\C)$
with corresponding integrator $(v, u),$ here $ v, u \in \C^{n}$,
then for any  $X\in GL_{n}(\C)$ it holds that
$\left(\begin{smallmatrix}
X & 0 \\
0 & 1
\end{smallmatrix}
\right)
A
\left(
\begin{smallmatrix}
X^{-1} & 0 \\
0 & 1
\end{smallmatrix}\right)$
is an integral of  $XBX^{-1}$ with corresponding integrator
$(vX^{-1}, uX^{\top})$.
\end{lemma}

This lemma allows one to reduce different questions concerning 
diagonalizable matrices to the case of diagonal ones. Therefore below we
restrict ourselves to the diagonal matrices.

\begin{notation} \label{N3.2}
Denote $\mathcal{B} =
{\rm diag\,}(\underbrace{b_1, \ldots, b_1}_{\alpha_1}, \ldots,
\underbrace{b_m, \ldots, b_m}_{\alpha_m}, a_1, \ldots, a_k) \in M_n(\K)$
with the characteristic polynomial $p_{\mathcal{B}} = f,$ from Notation \ref{DEF:f_q_Q_h}, i.e.
$$p_{\mathcal{B}} = (x-a_1) \ldots (x-a_k)
(x - b_1)^{\alpha_1} \ldots (x - b_m)^{\alpha_m}.$$
We denote $\mathcal{A} = \begin{pmatrix}
\mathcal{B} & u^{\top} \\
v & \tau(\mathcal{B})
\end{pmatrix}\in M_{n+1}(\K),$ where
$v = (v_1,\ldots,v_n), u = (u_1,\ldots,u_n) \in \K^{n}$.
We also define $C_0 = 0, \ C_i = C_{i-1} + \alpha_i,\ i = 1, \ldots, m$ and
$d_j = C_{m} + j,$ here $ j = 1,\ldots,k.$
\end{notation}

Below we always assume that $\mathcal{A}$ and $\mathcal{B}$
are as in Notation~\ref{N3.2}.

\begin{lemma} \label{LM:diagonal_pA} The following decomposition holds:
\begin{equation} \label{LINE:diagonal_pA}
p_{\mathcal{A}}(x) = (x - \tau(\mathcal{B}))p_{\mathcal{B}}(x) - 
\sum \limits_{i = 1}^{n}u_iv_i\frac{p_{\mathcal{B}}(x)}{x - \lambda_i},
\end{equation}
where $\lambda_i$ is the element of $\mathcal{B}$ located at position~$(i,i)$.
\end{lemma}
\begin{proof}
Follows from the Laplace decomposition of
$p_{\mathcal{A}}(x) =\det(xI - \mathcal{A})$
by the last column since $\mathcal{B}$ is diagonal.
\end{proof}

\begin{cor} \label{COR:diagonal_pA_preserves_multiple_roots}
Let  
$\lambda\in \K$ be
an eigenvalue of $\mathcal{B}$ of the multiplicity  $l>1$.
If $\mathcal{A}$ is an integral of $\mathcal{B}$ then $\lambda$ is an
eigenvalue of $\mathcal{A}$ of the multiplicity  $l + 1$ and
$\sum\limits_{j\,:\, \lambda_j=\lambda} u_jv_j=0$,
where $\lambda_j$ is the element of $\mathcal{B}$ located at position~$(j,j)$.
\end{cor}
\begin{proof}
We compute $p_{\mathcal{A}}(\lambda)$ by the formula \eqref{LINE:diagonal_pA}.
Each summand is zero since $\lambda$ is a multiple zero of
$p_{\mathcal{B}}=(x-\lambda_1)\ldots (x-\lambda_{n})$. 
Hence  $p_{\mathcal{A}}(\lambda) = 0$. Since
$p_{\mathcal{A}}'(x) = (n+1)p_{\mathcal{B}}$, it follows that
$\lambda$ is a zero of $p_{\mathcal{A}}$ of the multiplicity~$l+1$.

Thus $(x - \lambda)^{l + 1}\,|\,p_{\mathcal{A}}(x).$ In particular
$(x - \lambda)^{l}\,|\,p_{\mathcal{A}}(x).$ 

Since
$(x - \lambda)^{l}\,|\,p_{\mathcal{B}}(x)$ it follows
that $(x - \lambda)^{l}\,|\,(x - \tau(\mathcal{B}))p_{\mathcal{B}}(x)$
and for $\lambda_i \neq
\lambda$ it holds that
$(x - \lambda)^{l}\,|\,\frac{p_{\mathcal{B}}(x)}{x - \lambda_i}.$

Therefore by \eqref{LINE:diagonal_pA} we get
$(x - \lambda)^{l}\,|\,
\sum\limits_{j:\lambda_j =
\lambda}u_jv_j\frac{p_{\mathcal{B}}(x)}{x - \lambda}.$
Then from $(x - \lambda)^{l}\nmid\frac{p_{\mathcal{B}}(x)}{x - \lambda}$ it
follows that
$$\sum\limits_{j:\lambda_j = \lambda}u_jv_j = 0.$$
\end{proof}

\begin{lemma} \label{LM:pA_firs_coeff}
Two coefficients at the two highest degrees of $p_{\mathcal{A}}(x)$
do not depend on the choice of the vectors~$v$ and~$u$.
\end{lemma}
\begin{proof}
In the formula \eqref{LINE:diagonal_pA} the degrees of all summands
except the first
one do not exceed $\deg(p_{\mathcal{A}}) - 2$.
\end{proof}

\begin{lemma} \label{LM:diagonal_pA_simple_form}
Let $\mathcal{A}$
be an integral of $\mathcal{B}\in M_n(\K)$. Then
\begin{equation} \label{eq:3'}
p_{\mathcal{A}}(x) = (x - \tau(\mathcal{B}))p_{\mathcal{B}}(x) -
\sum \limits_{i = 1}^{k}u_{d_i}v_{d_i}
\frac{p_{\mathcal{B}}(x)}{x - a_i}. \end{equation} 
 
\end{lemma}
\begin{proof}
Separating the summands in the formula \eqref{LINE:diagonal_pA} into
the two sums corresponding to multiple and simple zeros, we have
by Lemma \ref{LM:diagonal_pA}  that
$$p_{\mathcal{A}}(x) = (x - \tau(\mathcal{B}))p_{\mathcal{B}}(x) +
\sum \limits_{i = 1}^{k}y_i\frac{p_{\mathcal{B}}(x)}{x - a_i} +
\sum \limits_{i = 1}^{m}z_i\frac{p_{\mathcal{B}}(x)}{x - b_i},$$

where $y_i = -u_{d_i} \cdot v_{d_i},\ i = 1,\ldots, k$ and
$z_i = \sum\limits_{j = C_{i-1} + 1}^{C_i}-u_jv_j,\ i = 1,\ldots,m$.

By Corollary \ref{COR:diagonal_pA_preserves_multiple_roots} we obtain
that $z_i = 0,\ i = 1,\ldots,m.$ Therefore
$$p_{\mathcal{A}}(x) = (x - \tau(\mathcal{B}))p_{\mathcal{B}}(x) +
\sum \limits_{i = 1}^{k}y_i\frac{p_{\mathcal{B}}(x)}{x - a_i}.$$
\end{proof}

\begin{cor} \label{COR:integrator_with_smallest_norm}
Let $\K = \C$ and $\mathcal{A}$ be an integral of $\mathcal{B}$ and let  
$$\mathcal{A'} = \begin{pmatrix}
\mathcal{B} & u'^{\top} \\
v' & \tau(\mathcal{B})
\end{pmatrix}, \text{ where }
u_i' = v_i' = \begin{cases}
\sqrt{u_iv_i}, \  i = d_1, \ldots, n,\\
0,    \          i = 1, \ldots, C_m.
\end{cases}
$$
Then 

1. $\mathcal{A'}$
is also an integral
of $\mathcal{B}$.

2. $
||\mathcal{A'}||_{F}^{2} = ||\mathcal{B}||_{F}^{2} + |\tau(\mathcal{B})|^2 +
2\sum\limits_{i = 1}^{k} |u_{d_i}v_{d_i}|.
$

3. For any integral $\mathcal{A''} = \begin{pmatrix}
\mathcal{B} & u''^{\top} \\
v'' & \tau(\mathcal{B})
\end{pmatrix}$ of $\mathcal{B}$ with
$p_{\mathcal{A}} = p_{\mathcal{A''}}$ it holds that
$||\mathcal{A'}||_{F}^{2}  \leq
||\mathcal{A''}||_{F}^{2}$.

\end{cor}
\begin{proof} 1. Applying the formula \eqref{eq:3'}
we get
$$
p_{\mathcal{A'}}(x) = (x - \tau(\mathcal{B}))p_{\mathcal{B}}(x) - 
\sum \limits_{i = 1}^{k}u_{d_i}'v_{d_i}'
\frac{p_{\mathcal{B}}(x)}{x - a_i} =
(x - \tau(\mathcal{B}))p_{\mathcal{B}}(x) - 
\sum \limits_{i = 1}^{k}u_{d_i}v_{d_i}
\frac{p_{\mathcal{B}}(x)}{x - a_i} = p_{\mathcal{A}}(x).
$$
Therefore $\mathcal{A'}$ is an integral of $\mathcal{B}.$

2. Let us compute the Frobenius norm of $\mathcal{A'}$ by the definition
of a norm and taking into account that $\mathcal{B}$ is a
submatrix of $\mathcal{A'}$
and definition of $u'_i,v'_i$
$$
||\mathcal{A'}||_{F}^{2} = ||\mathcal{B}||_{F}^{2} + |\tau(\mathcal{B})|^2 +
||u'||^2 + ||v'||^2 = ||\mathcal{B}||_{F}^{2} + |\tau(\mathcal{B})|^2 +
2\sum\limits_{i = 1}^{k} |u_{d_i}  v_{d_i}|.
$$

3.
Since ${\mathcal{A''}}$ is an integral of ${\mathcal{B}}$,
the equality \eqref{eq:3'} implies
$$p_{\mathcal{A''}}(x) = (x - \tau(\mathcal{B}))p_{\mathcal{B}}(x) - 
\sum \limits_{i = 1}^{k}u_{d_i}''v_{d_i}''
\frac{p_{\mathcal{B}}(x)}{x - a_i}.$$
By the conditions   $p_{\mathcal{A}} = p_{\mathcal{A''}}$. It follows that
$$
u_{d_i}''v_{d_i}''\frac{p_{\mathcal{B}}}{x - a_i}(a_i) =
-p_{\mathcal{A''}}(a_i) = -p_{\mathcal{A}}(a_i) =
u_{d_i}v_{d_i}\frac{p_{\mathcal{B}}}{x - a_i}(a_i), \ i = 1,\ldots, k.
$$
Therefore
$u''_{d_i}v''_{d_i} = u_{d_i}v_{d_i}, \ i = 1, \ldots, k.$ Observe that   
$$
||\mathcal{A''}||_{F}^{2} = ||\mathcal{B}||_{F}^{2} + |\tau(\mathcal{B})|^2 +
||u''||^2 + ||v''||^2 \geq ||\mathcal{B}||_{F}^{2} + |\tau(\mathcal{B})|^2 +
\sum\limits_{i = 1}^{k} |u''_{d_i}|^2 +
\sum\limits_{i = 1}^{k} |v''_{d_i}|^2.
$$
Combining with the item 2 we obtain that  to prove
$||\mathcal{A'}||_{F}^{2} \leq ||\mathcal{A''}||_{F}^{2}$
it is sufficient to show 
$|u''_{d_i}|^2 + |v''_{d_i}|^2 \geq 2 |u_{d_i}v_{d_i}|,$
which holds because
$$
0 \leq (|u''_{d_i}| - |v''_{d_i}|)^2 =
|u''_{d_i}|^2 + |v''_{d_i}|^2 - 2 |u''_{d_i}v''_{d_i}| =
|u''_{d_i}|^2 + |v''_{d_i}|^2 - 2 |u_{d_i}v_{d_i}|.
$$
\end{proof}

The following statement summarizes our previous study.

\begin{theorem} \label{THM:integrable_iff_fullintegral}
$\mathcal{B}$ is integrable if and only if
$p_{\mathcal{B}}(x)$ has a full integral.
\end{theorem}
\begin{proof}
Let us prove the necessity. Let $\mathcal{B}$ be integrable and $\mathcal{A}$
be an integral of $\mathcal{B}$.
Then by definition $p_{\mathcal{A}}' = (n + 1)p_{\mathcal{B}}$ and
by Lemma \ref{LM:diagonal_pA_simple_form}
\begin{equation} \label{EQ:pA}
p_{\mathcal{A}} = (x - \tau(\mathcal{B}))p_{\mathcal{B}} +
\sum\limits_{i = 1}^{k} w_i\frac{p_{\mathcal{B}}}{x - a_i}, \ \ w_i \in \K.
\end{equation}
Substituting $x=b_i$ to the formula \eqref{EQ:pA} one has that
$p_{\mathcal{A}}(b_i) = 0$ for all $ i = 1, \ldots ,m$.
Thus  $F := \frac{1}{n + 1}p_{\mathcal{A}}$ is a full
integral of $p_{\mathcal{B}}$.

Let us prove the sufficiency. Assume now that there exists a full
integral $F$ of $p_{\mathcal{B}}$.
Let us show that there exist $v_{d_1},\ldots, v_{d_k}\in \K$ such that
$\mathcal{A}$ is an integral of $\mathcal{B}$, where the couple of vectors
$v = (\underbrace{0, \ldots, 0}_{C_m}, v_{d_1},  \ldots ,
v_{d_k})$ and $ u = (1,  \ldots , 1)$
is the corresponding integrator.

From the formula for $p_{\mathcal{A}}$ and the definition for $F$
we have that $q$ divides
$p_{\mathcal{A}}$ and $F$. Recalling $h,q$ from Notation \ref{DEF:f_q_Q_h}
we denote  
$$\tilde p_{\mathcal{A}} := \frac{p_{\mathcal{A}}}{q}, \ \tilde F :=
\frac{(n + 1)F}{q}, \ 
h_{a_i} := \frac{h}{x - a_i},  \ 
g := p_{\mathcal{A}} - (n + 1)F,\ i = 1,\ldots,k.$$

Consider the equation $p_{\mathcal{A}} = (n+1)F.$ 
If we take
$v = (\underbrace{0, \ldots, 0}_{C_m}, v_{d_1},  \ldots , v_{d_k})$
and $ u = (1,  \ldots , 1)$
then this becomes an equation with $k$ variables
$v_{d_1},\ldots,v_{d_k}.$
We now show that
$v_{d_i} = \frac{(n + 1)F(a_i)}{h_{a_i}(a_i)}, \ i = 1,\ldots,k$
is the solution for this equation.
By the direct substitution of  $a_i$ and the chosen values of $u$ and $v$ into the formula for
$p_{\mathcal{A}}$ we obtain that 
$p_{\mathcal{A}}(a_i) = v_{d_i} \cdot  h_{a_i}(a_i),\ i = 1,\ldots,k.$ So
$$g(a_i) = p_{\mathcal{A}}(a_i) - (n + 1)F(a_i) =
v_{d_i} h_{a_i}(a_i) - (n + 1)F(a_i) =  0, \ i = 1,\ldots,k.$$
By Lemma \ref{LM:pA_firs_coeff} and the
definition of  $F$  we obtain that the
coefficients at monomials $x^{n+1}$ and  $x^{n}$ in polynomials 
$p_{\mathcal{A}}$ and 
$(n + 1)F$ are equal. Therefore $\deg(g) \leq n - 1.$
Since  $q(a_i) \neq 0,$ $i = 1, \ldots ,k$, it follows from
$$0 = g(a_i) = q(a_i)(\tilde p_{\mathcal{A}}(a_i) - \tilde F(a_i)), \
i = 1, \ldots ,k,$$
  that
$\tilde p_{\mathcal{A}}(a_i) = \tilde F(a_i), \ i = 1, \ldots ,k$, 
and since
$\deg(\tilde p_{\mathcal{A}} - \tilde F) \leq n - 1 - \deg(q) = k - 1,$
then  $\tilde p_{\mathcal{A}} = \tilde F$. Thus
$p_{\mathcal{A}} = q\tilde p_{\mathcal{A}} = q\tilde F = (n + 1)F.$
Hence for  $v_{d_i} = \frac{(n + 1)F(a_i)}{h_{a_i}(a_i)},\ i = 1,\ldots,k$
we get $$
p_{\mathcal{A}}'(x) = (n + 1)F' = (n + 1)p_{\mathcal{B}}.$$
\end{proof}

\begin{cor} \label{COR:pA_is_full_int}
Let $\mathcal{A}$ be an integral of $\mathcal{B}$. Then
$\frac{1}{n + 1}p_{\mathcal{A}}$ is a full integral of~$p_{\mathcal{B}}.$
\end{cor}
\begin{proof}
Directly shown at the end of the proof of sufficiency of
Theorem~\ref{THM:integrable_iff_fullintegral}.
\end{proof}

\begin{cor} \label{COR:integrator_formula}
If $p_{\mathcal{B}}(x)$ has a full integral $F(x)$ then an integrator
of $\mathcal{B}$ can be chosen as follows 
$u_i = 1, \ i = 1, \ldots, n,\  v_1 = \ldots = v_{C_m} = 0, \
v_{d_i} =  \frac{(n+1)F(a_i)}{h_{a_i}(a_i)}, \ i = 1, \ldots, k.$
In this case $p_{\mathcal{A}}(x) = (n+1)F(x).$
\end{cor}
\begin{proof}
Directly shown in the proof of Theorem \ref{THM:integrable_iff_fullintegral}.
\end{proof}

\begin{rmk}
Corollary~\ref{COR:integrator_formula} does not describe all
possible integrators and corresponding integrals. For example, if $(u, v)$
is an integrator of $\mathcal{B}$, then for any $s \in \K\setminus\{0\}$
the pair of vectors $(su, s^{-1}v)$ is also an integrator of $\mathcal{B},$
which is not described by Corollary~\ref{COR:integrator_formula}.
Indeed, as it is shown in
Lemma \ref{LM:diagonal_pA} the characteristic
polynomial   depends only
on the products of the  coordinates of the vectors $u$ and $v$ with the equal
indices. The integral $\mathcal{A}$ is determined by the choice of the
integrators~$u,v$.
\end{rmk}

\begin{cor} \label{COR:integrator_formula_with_smallest_norm}
Let $p_{\mathcal{B}}(x)$ has a full integral $F(x)$. Then the formula
$$
u_i = v_i = \begin{cases} 0,       \ \       i = 1, \ldots, C_m, \\
\sqrt{\frac{(n+1)F(a_i)}{h_{a_i}(a_i)}}, i = C_m+1, \ldots, n.
\end{cases}
$$
for  integrators determines the integral
$\mathcal{A}$
of $\mathcal{B}$ with $p_{\mathcal{A}}(x) = (n+1)F(x)$ such that its 
Frobenius norm is the least possible.  
In this case $||\mathcal{A}||_{F}^{2} = ||\mathcal{B}||_{F}^{2} +
|\tau(\mathcal{B})|^2 + 2\sum\limits_{i = 1}^{k}
\left|\frac{(n+1)F(a_i)}{h_{a_i}(a_i)}\right|.$
\end{cor}
\begin{proof}
Direct application of Corollary \ref{COR:integrator_with_smallest_norm} and
Corollary \ref{COR:integrator_formula}.
\end{proof}

\begin{theorem} \label{THM:diagonal_integrable_classifications}
Let $\mathcal{B}$ be a diagonal matrix introduced in Notation \ref{N3.2}. Then\\
1) if $m \le 1$ then the matrix  $\mathcal{B}$ has an integrator,\\
2) if $m > k + 1$ then the matrix  $\mathcal{B}$ does not have an integrator,\\
3) in the other cases the existence of integrators depends 
on the values of the eigenvalues of $\mathcal{B}$,  i.e. for any sequence of
multiplicities  there are  eigenvalues 
for which an integrator exists  and there are eigenvalues
for which integrator does not exist.
\end{theorem}
\begin{proof}
By Theorem \ref{THM:integrable_iff_fullintegral} the integrability of a
matrix is equivalent to the full integrability of its characteristic polynomial.
Then  Theorem \ref{THM:full_int_classification} is applicable and concludes
the proof.
\end{proof}

\begin{rmk}
The subset of integrable matrices is dense in $M_n({\mathbb C})$ and
the subset of non-integrable matrices is sparse. 
\end{rmk}
\begin{proof}
Indeed, the subset of non-derogatory diagonalizable matrices is dense,
by the first item of
Theorem \ref{THM:diagonal_integrable_classifications} such matrices are
integrable. The complement to the subset of non-derogatory matrices is sparse,
therefore the subset of non-integrable matrices is sparse.
\end{proof}

\begin{lemma}
Let $m > 1$ and $q(x) \in \C[x]$ be fixed.  Denote by
$S \subseteq M_n(\C)$  the subset of matrices such that $q$ is a factor of their
characteristic polynomials.
Then the subset of non-integrable matrices $S_{1} \subseteq S$ is dense in $S$
and the subset of integrable matrices $S_{2} \subseteq S$ is sparse in $S.$
\end{lemma}
\begin{proof}
By Lemma \ref{LM:integratordoesnotexists_helper} if
$f \notin q\cdot\text{Im}\,\varphi_{k-m+1,m}$ then $f$ does not
possess a full integral.

Since
$$\dim \text{Im}\,\varphi_{k-m+1,m} = k - m + 2 < k+1 = \dim \C_{k}[x],$$
one has $\text{Im}\,\varphi_{k-m+1,m}$ is sparse in $\C_{k}[x]$ and
$q\cdot\text{Im}\,\varphi_{k-m+1,m}$ is sparse in~$q\cdot\C_{k}[x].$

Consider the map
$$\varrho: S \longrightarrow q\cdot\C_k[x],$$
$$\varrho(M) = p_{M}(x).$$
Since $\varrho$ is continuous then
$\varrho^{-1}(q\cdot\text{Im}\,\varphi_{k-m+1,m})$ is sparse in~$S$
as a preimage of sparse subset under the action of the surjective continuous
map~$\varrho.$
Hence $S_2 \subset\varrho^{-1}(q\cdot\text{Im}\,\varphi_{k-m+1,m})$
is sparse in~$S$.

Since $S = S_1 \cup S_2$ then $S_1$ is dense in~$S.$
\end{proof}

\section{Diagonalizability of the integral}

\begin{theorem}
\label{THM:diagonal_integral_diagonalizable_criteria}
Let $\mathcal{A}$
be an integral of
$\mathcal{B}$. Then $\mathcal{A}$ is diagonalizable
if and only if the following two conditions are satisfied simultaneously
for the integrators~$u,v$:

1. $u_1 = v_1 = \ldots = u_{C_m} = v_{C_m} = 0$,

2. $v_{d_i} = u_{d_i} = 0$ for any $i$ such that
$p_{\mathcal{A}}(a_i) = 0$.

Here we use the notations for coordinates and indices
introduced in Notation~\ref{N3.2}.
\end{theorem}
\begin{proof}
By Corollary \ref{COR:pA_is_full_int} the characteristic
polynomial $p_{\mathcal{A}}(x)$
is a full integral of $\frac{1}{n+1}p_{\mathcal{B}}(x).$
Hence, $p_{\mathcal{A}}(x) = Q(x)H(x)$ for some  $H(x)\in \K[x]$,
here $Q(x)$ is defined by Notation~\ref{DEF:f_q_Q_h}.

1. The multiplicities of the zeros of $H(x)$ are less than or equal to 2.
Indeed, the zeros of $H(x)$ of the multiplicity greater than 2 are
the multiple zeros of $p_{\mathcal{B}}(x)$. But all the multiple zeros of
$p_B(x)$ are included into $Q(x),$ and thus cannot be the zeros of $H(x)$.

2. Let us prove the necessity. Assume that the conditions 1 and 2 are satisfied.
To show that $\mathcal{A}$ is diagonalizable we
calculate  $\dim \text{Ker} (\mathcal{A} - \lambda I)$ for all multiple
eigenvalues $\lambda$ of~$A$ in order to show that the geometric multiplicity of
each eigenvalue coincides with its algebraic multiplicity.
The general situation splits into the following two cases since
the only multiple zeros of $p_{\mathcal{A}}(x)$ are the
zeros of $p_{\mathcal{B}}(x)$.

2.1. $\lambda = b_i, \ i = 1, \ldots ,m.$ Without loss of generality we assume
that $i = 1.$ As the multiplicity of $\lambda$ in $\mathcal{A}$
is $\alpha_1 + 1,$ we need to
show that $\dim \text{Ker} (\mathcal{A} - \lambda I) = \alpha_1+1.$  
For any $j = 1, \ldots, \alpha_1$ the vector
$e_j  
\in \text{Ker} (A - \lambda I)$, since 
$$(A-\lambda I)e_j=\left(O_{\alpha_1}\oplus \bigoplus\limits_{i=2}^m
(b_i-b_1)I_{\alpha_i} \oplus 
\left(\begin{array}{ccccc}   a_1 - b_1& 0 &\ldots & 0&u_{d_1}\\
0 & \ddots& \ddots&\vdots&\vdots\\
\vdots & \ddots &  \ddots& 0 & \vdots \\
0 & \ldots &0 & a_k - b_1& u_n\\
v_{d_1}&\ldots & \ldots & v_n & \tau(\mathcal{B})\\
\end{array}
\right)\right) e_j=0.$$

Therefore $\dim \text{Ker} (\mathcal{A} - \lambda I) \ge \alpha_1.$ 

Consider the submatrix $A'\in M_{k+1}(\K)$ of $\mathcal{A}$ such that
$\mathcal{A}={\rm diag\,}\{\underbrace{b_1,\ldots, b_1}_{\alpha_1}, \ldots,
\underbrace{b_m, \ldots, b_m}_{\alpha_m}\}$ $ \oplus A'$.
Then $p_{\mathcal{A}}(x) = q(x)p_{A'}(x).$
Since $(x - b_1)^{\alpha_1 + 1} \mid p_{\mathcal{A}}(x)$, it follows that
$p_{A'}(b_1) = 0.$ Hence $\det(A' - \lambda I) = 0$ and there
exists a vector $w = (w_1,\ldots, w_{k + 1})$ such that
$(A' - b_1I)w^\top = 0.$ This provides the $(\alpha_1+1)-$st 
vector in $\text{Ker} (\mathcal{A} - \lambda I)$, i.e.:
$$(\mathcal{A} - \lambda I)
(\underbrace{0, \ldots, 0}_{C_m}, w_1, \ldots, w_{k + 1})^{\top} = 0.$$
It is straightforward to see that $\{e_1,\ldots,e_{\alpha_1},w\}$ is a linearly
independent system of vectors. Thus
$\dim \text{Ker} (\mathcal{A} - \lambda I)=\alpha_1+1$ equals to
the algebraic multiplicity of $\lambda$. Hence the Jordan block
corresponding to $\lambda$ is diagonal.

2.2. $\lambda \notin \{b_1,  \ldots , b_m\}.$
It is shown in the item 1 that in this case
the multiplicity of $\lambda$ is 2. Thus
$\lambda = a_j$ for some $j = 1,\ldots, k.$ Without loss of
generality we assume $j = 1.$ Then by the assumptions
we have $v_{d_1} = u_{d_1} = 0.$ Thus
\begin{equation} \label{eq:v_2} (\mathcal{A} - \lambda I)
(\underbrace{0, \ldots, 0}_{C_m}, 1, 0, \ldots, 0)^{\top} = 0.\end{equation}
To construct the other vector we consider the submatrix
$A''\in M_{k}(\K)$  of
$\mathcal{A}$  such that
$$\mathcal{A}={\rm diag\,}\{\underbrace{b_1,\ldots, b_1}_{\alpha_1}, \ldots,
\underbrace{b_m, \ldots, b_m}_{\alpha_m}, a_1 \} \oplus A''.$$ 
Then $p_{\mathcal{A}}(x) = q(x)(x - a_1)p_{A''}(x).$
Since $(x - a_1)^{2} \mid p_{\mathcal{A}}(x)$ then $p_{A''}(a_1) = 0.$
Hence there exists a vector $z=( z_1, \ldots, z_{k})$ such that
$(A'' - \lambda I)z^{\top} = 0.$ Then
$(\mathcal{A} - \lambda I)
(\underbrace{0, \ldots, 0}_{C_m + 1}, z_1, \ldots, z_{k})^{\top} = 0.$
Together with the equality \eqref{eq:v_2} this implies
  $\dim \text{Ker}(\mathcal{A} - \lambda I) = 2.$

3. Let us prove the sufficiency. Assume that some coordinate of $u$
and $v$ corresponding to the common eigenvalues of $\mathcal{A}$ and
$\mathcal{B}$ is nonzero. We can consider two cases:

3.1. The common eigenvalue of $\mathcal{A}$ and
$\mathcal{B}$ is a multiple eigenvalue of $\mathcal{B}$.
Since all the multiple zeros of
$p_{\mathcal{B}}$ are the zeros of $p_{\mathcal{A}}$, it follows that $m > 0$.
Without loss of generality $u_1 \neq 0.$
Consider
the vector
$w' = \mu_{C_1 + 1} r_{C_1 + 1} + \mu_{C_1 + 2} r_{C_1 + 2} + \ldots +
\mu_n r_n + \mu_1r_1,$
where $r_i$ is the $i-$th row of $(\mathcal{A} - b_1I).$
Then for its coordinates we have
$$\begin{cases}
w_{1}' = w_{2}' = \ldots = w_{C_1}' = 0,\\
w_{C_1 + i}' = \mu_{C_1 + i}(b_2 - b_1),\ i = 1, \ldots, \alpha_1,\\
w_{C_2 + i}' = \mu_{C_2 + i}(b_3 - b_1),\ i = 1, \ldots, \alpha_2,\\
\cdots\\
w_{C_{m - 1} + i}' = \mu_{C_{m - 1} + i}(b_{m} - b_1),\ 
i = 1, \ldots, \alpha_m,\\
w_{d_i}' = \mu_{d_i}(a_i - b_1),\ i = 1, \ldots, k.
\end{cases}$$
If $w' = 0$ then $\mu_i = 0, \  i = C_1, \ldots, n.$
In this case $w_{n + 1}' = \mu_{1}u_{1},$ so $\mu_{1} = 0.$
This means that $w' = 0$  if and only if  $\mu_1 = \ldots = \mu_n = 0.$
Therefore the first row of
$(\mathcal{A} - b_1 I)$  and the rows with the indices
$\alpha_1 + 1, \ldots, n$ form a linearly independent set. 
Hence the rank of $(\mathcal{A} - b_1I)$ is at
least $n - \alpha_1 + 1.$ Therefore
$\dim \text{Ker}(\mathcal{A} - b_1I) =
n + 1 - rk(\mathcal{A}) \leq \alpha_1 < \alpha_1 + 1.$
Thus $\mathcal{A}$ is not diagonalizable.

3.2. The common eigenvalue of $\mathcal{A}$ and
$\mathcal{B}$ is a simple eigenvalue of $\mathcal{B}$.
Then $k > 0$.
Without loss of generality $u_{d_1} \neq 0,$ $p_{\mathcal{A}}(a_1) = 0.$
Consider
the vector $w'' = \mu_1 r_1 + \mu_2 r_2 + \ldots + \mu_n r_n,$
where $r_i$ is the $i-$th row of $(\mathcal{A} - a_1I).$
Then for its coordinates we have
$$\begin{cases}
w_{d_1}'' = 0,\\
w_{d_i}'' = \mu_{d_i}(a_i - a_1),\ i = 2, \ldots, k,\\
w_{C_0 + i}'' = \mu_{C_0 + i}(b_1 - a_1),\ i = 1, \ldots, \alpha_1,\\
w_{C_1 + i}'' = \mu_{C_1 + i}(b_2 - a_1),\ i = 1, \ldots, \alpha_2,\\
\cdots\\
w_{C_{m - 1} + i}'' = \mu_{C_{m - 1} + i}(b_{m} - a_1),\ 
i = 1, \ldots, \alpha_m.
\end{cases}$$
If $w'' = 0$ then $\mu_i = 0, \ i = 1, \ldots, n, \ i \neq d_1.$
In this case $w_{n + 1}'' = \mu_{d_1}u_{d_1},$ thus $\mu_{d_1} = 0.$
This means that $w = 0$ if and only if $\mu_1 = \ldots = \mu_n = 0.$
Therefore the first $n$ rows of
$(\mathcal{A} - a_1I)$
form  linearly independent set.
So the rank of $(\mathcal{A} - a_1I)$
is at least $n.$ Hence
$\dim \text{Ker}(\mathcal{A} - a_1I) =  n + 1 - rk(\mathcal{A})  < 2.$
Thus $\mathcal{A}$ is not diagonalizable.
\end{proof}

\begin{cor}
Let $\mathcal{B}$ be an integrable diagonalizable matrix.
Then among the integrals of $\mathcal{B}$ there are
both non-diagonalizable matrices and  diagonalizable matrices.
\end{cor}
\begin{proof}
Let $\mathcal{A} = \begin{pmatrix}
\mathcal{B} & u^{\top}\\
v & \tau(\mathcal{B})
\end{pmatrix}$ be an integral of $\mathcal{B}$.\\
1) Assume that $\mathcal{A}$ has at least one common
eigenvalue with $\mathcal{B}$.
For given vectors $v, u$ consider the vectors
$$v' = (1, \ldots, 1), u' = (u_1',\ldots, u_n'), \text{ where } 
u_i' = \begin{cases}
\frac{u_i}{v_i}, \ v_i \neq 0;\\
0, \ v_i = 0.
\end{cases}$$
From Lemma \ref{LM:diagonal_pA} we obtain that $\mathcal{A}' = \begin{pmatrix}
\mathcal{B} & u'^{\top}\\
v' & \tau(\mathcal{B})
\end{pmatrix}$ is an integral of $\mathcal{B},$ since
$p_{\mathcal{A}}(x) = p_{\mathcal{A}'}(x).$ By
Theorem \ref{THM:diagonal_integral_diagonalizable_criteria}
$\mathcal{A}'$ is not diagonalizable.\\
1.1) If $\mathcal{B}$ has a multiple eigenvalue, then any integral of
$\mathcal{B}$ has a common eigenvalue with $\mathcal{B}$ due to Corollary
\ref{COR:diagonal_pA_preserves_multiple_roots}.
Thus by the item 1 one can construct a non-diagonalizable integral of
$\mathcal{B}.$\\
1.2) Otherwise $\mathcal{B}$ is non-derogatory, so by
\cite[Theorem 9]{IntegratorsOfMatricies} for any $t \in \K$
there exists such integral $\mathcal{A}_{t}$ of $\mathcal{B}$ that
$p_{\mathcal{A}_{t}}(x) = p_{\mathcal{A}}(x) - t.$ Denote by $\lambda$
some eigenvalue of
$\mathcal{B}.$ Then $\lambda$ is an eigenvalue of
$\mathcal{A}_{p_{\mathcal{A}}(\lambda)}.$
Thus $\mathcal{A}_{p_{\mathcal{A}}(\lambda)}$ is
an integral of $\mathcal{B}$ with a
common eigenvalue with $\mathcal{B}$, therefore by
the item 1 one can construct a non-diagonalizable integral of $\mathcal{B}.$\\
2) For given $v, u$ consider vectors
$v'' = (v_1'',\ldots, v_n''), u'' = (u_1'',\ldots, u_n''),$ where 
$$u_i'' = \begin{cases}
0,\ v_iu_i = 0;\\
0, \ u_i \ \text{corresponds to a multiple eigenvalue}; \\
u_i, \ \text{otherwise};
\end{cases}$$$$
v_i'' = \begin{cases}
0, \ v_iu_i = 0;\\
0, \ v_i \ \text{corresponds to a multiple eigenvalue}; \\
v_i. \ \text{otherwise}
\end{cases}$$\\
From Lemma \ref{LM:diagonal_pA} and
Lemma \ref{COR:diagonal_pA_preserves_multiple_roots}
we obtain that
$\mathcal{A}'' = \begin{pmatrix}
\mathcal{B} & u''^{\top}\\
v'' & \tau(\mathcal{B})
\end{pmatrix}$
is an integral of $\mathcal{B},$ since
$p_{\mathcal{A}}(x) = p_{\mathcal{A}''}(x).$
Let $\lambda$ be a simple eigenvalue of $\mathcal{B},$ which is also
an eigenvalue of $\mathcal{A}$. Then by Lemma \ref{LM:diagonal_pA} we obtain
$$
0 = p_{\mathcal{A}}(\lambda) =
-u_iv_i\frac{p_{\mathcal{B}}}{(x - \lambda)}(\lambda), \ \text{for some}\ i.
$$
Since $\frac{p_{\mathcal{B}}}{(x - \lambda)}(\lambda) \neq 0,$ then
$u_iv_i = 0,$ so $u_i''=v_i'' = 0.$ Thus any coordinate of $v'', u''$
that corresponds to a common eigenvalue of $\mathcal{A}''$ and $\mathcal{B}$ is
equal to 0. Hence by Theorem \ref{THM:diagonal_integral_diagonalizable_criteria}
$\mathcal{A}''$ is diagonalizable.
\end{proof}
\begin{cor} \label{COR:seq_int_iff_seq_fint}
There exists a sequence of diagonalizable matrices
$B = B_1, B_2, \ldots,$ such that $B_{i + 1}$ is an integral of
$B_i, i \in \mathbb N$ if and only if there exists a sequence of polynomials
$p_{B}(x) = p_1(x), p_2(x), \ldots,$ such that $p_{i+1}(x)$
is a full integral of $p_{i}(x), i \in \mathbb N.$
\end{cor}
\begin{proof}
1. Let us prove the necessity. If the sequence $B = B_1, B_2, \ldots,$
such that $B_{i + 1}$ is an integral of
$B_i, \ i \in \mathbb N,$ exists then by
Corollary \ref{COR:pA_is_full_int} the sequence
$p_{B_1}(x), \frac{1}{\deg p_{B_1}(x)}p_{B_2}(x),$ $
\frac{1}{\deg p_{B_2}(x)}p_{B_3}(x),
\ldots$ has the desired property.

2. Let us prove the sufficiency. If the sequence
$p_{B}(x) = p_1(x), p_2(x), \ldots,$ such that $p_{i+1}(x)$
is a full integral of $p_{i}(x), i \in \mathbb N,$ exists then by taking an
diagonalizable integral we obtain the sequence $B_1, B_2, \ldots$
with the desired property.
\end{proof}

Let us remind that $k$ is the number of the simple zeros of
a polynomial $p$, and $m$ is the number of its different multiple roots,
in accordance with Definition \ref{DEF:km}.
\begin{lemma}
Let $m = 0.$ Then there exists a sequence of diagonalizable matrices
$\mathcal{B} = B_1, B_2, \ldots,$ where $B_{i + 1}$ is an integral of
$B_i, \ i \in \mathbb N.$
\end{lemma}
\begin{proof}
If $m = 0$ then $\mathcal{B}$ is integrable by
Theorem \ref{THM:diagonal_integrable_classifications}.
Let $\mathcal{A}$ be an integral of $\mathcal{B}.$ Since $\mathcal{B}$ is
non-derogatory then by \cite[Theorem 9]{IntegratorsOfMatricies}
for any $t \in \K$
there exists an integral $\mathcal{A}_{t}$ of $\mathcal{B}$ such that
$p_{\mathcal{A}_{t}}(x) = p_{\mathcal{A}}(x) - t.$ Taking $t$ different from
$p_{\mathcal{A}}(\lambda), \ \lambda \in \text{spec}(\mathcal{B})$
we obtain that
$p_{\mathcal{A}_{t}}(\lambda) \neq 0,$ thus $p_{\mathcal{A}_{t}}(x)$
has no multiple zeros and
$\mathcal{A}_{t}$ is non-derogatory.
Therefore for any non-derogatory diagonalizable matrix there
exists a non-derogatory diagonalizable integral.
Thus we can construct the desired sequence.
\end{proof}
\begin{lemma}
Consider a sequence of diagonalizable matrices
$\mathcal{B} = B_1, B_2, \ldots, B_{l},$ where $B_{i + 1}$ is an integral of
$B_i$ for each $ i = 1, \ldots, l - 1$. Then $m \leq 1 + \frac{k}{l - 1}.$
\end{lemma}
\begin{proof}
Let $F(x)$ be a full integral of
$p_{\mathcal{B}}(x)=
(x - a_1)\ldots(x - a_k)(x - b_1)^{\alpha_1}\ldots(x - b_m)^{\alpha_m},$
where $a_1,\ldots, a_k, b_1, \ldots, b_m\in \K$ are pair-wise distinct and
$\alpha_1,\ldots, \alpha_m \in
\mathbb N \setminus \{1\},$ i.e., $p_{\mathcal{B}}(x)$ has $k$ simple zeros
and $m$ multiple zeros. Then
$(x - b_1)^{\alpha_1 + 1}\ldots(x - b_m)^{\alpha_m + 1}\ |\ F(x).$
Thus $F(x)$ has at least $m$ multiple
zeros. Since $\deg F(x) = \deg p_{\mathcal{B}}(x) + 1$ then $F(x)$ has
at most $k + 1 - m$ zeros different from $b_1, \ldots, b_m$. Hence $F(x)$
has at most $k + 1 - m$ simple roots. If $m > 1$ then
$k + 1 - m < k$. Hence $F(x)$ has less simple
zeros than $p_{\mathcal{B}}(x)$. Therefore the number of simple zeros of
$p_{B_l}(x)$ is at most $k - (m - 1)(l - 1)$ and the number of multiple zeros
is at least $m$. Therefore by Theorem
\ref{THM:diagonal_integrable_classifications} we have
$$k - (m - 1)(l - 1) \geq m - 1,$$
$$k \geq (l - 1)(m - 1),$$
$$m \leq 1 + \frac{k}{l-1}.$$
\end{proof}

\begin{lemma}
Let $m = 1$ and $k < 2.$ Then there exists a sequence of diagonalizable matrices
$\mathcal{B} = B_1, B_2, \ldots,$ where $B_{i + 1}$ is an integral of
$B_i,\ i \in \mathbb N.$
\end{lemma}
\begin{proof}
1. If $k = 0$ then $\mathcal{B} = \lambda I_n$ and the sequence
$\lambda I_n, \lambda I_{n+1}, \ldots\ $ satisfies the conditions of the lemma.

2. If $k = 1$ then $p_{\mathcal{B}} = (x - a)(x - b)^{n-1}.$
It is straightforward to see that
$F_1 =\frac{1}{n+1}(x - \lambda_1)(x - b)^{n},$
where $\lambda_1 = a + \frac{a - b}{n},$ is a  full integral of
$p_{\mathcal{B}}.$ Similarly,
$F_2 = \frac{1}{(n+1)(n+2)}(x - \lambda_2)(x - b)^{n+1},$
where $\lambda_2 = \lambda_1 + \frac{\lambda_1 - b}{n+1}.$
Thus we obtain the sequence of polynomials
$  F_0=p_{\mathcal{B}}, F_1, F_2, \ldots$, where $F_i$ is a full
integral of $F_{i - 1}, \ i \in \mathbb N.$
By Corollary \ref{COR:seq_int_iff_seq_fint} we obtain the required sequence
$\mathcal{B} = B_1, B_2, \ldots.$
\end{proof}

\begin{rmk}
If $m = 1$ and $k \geq 2$ then the integral $\mathcal{A}$ of $\mathcal{B}$
can be non-integrable. For example, if
$p_{\mathcal{B}}(x) = x^2(x - 3)(x - 5)$ then it is straightforward to
check that
$F(x) = \frac{1}{5}x^3(x - 5)^2$ is the only full integral of
$p_{\mathcal{B}}(x)$. Thus $p_{\mathcal{A}}(x) = 5F(x)$
by Corollary \ref{COR:pA_is_full_int}. Hence
by Theorem \ref{THM:diagonal_integrable_classifications}
$\mathcal{A}$ is not integrable.
\end{rmk}

\bigskip
\noindent
\section{Applications to dual Schoenberg type inequality}
Sendov's conjecture for polynomials was first formulated
in 1958. It was then mentioned in Hayman's famous research
problems book \cite{Hayman}.
 
\noindent
{\bf Sendov conjecture (1958):} Let $p$ be a polynomial of degree $n \ge 2$
with zeros $z_1, ... , z_n$ and critical points
$w_1, ... , w_{n-1}$. Then, 
$$
\max\limits_{1\le k\le n} \min\limits_{1\le i\le n-1} \left | w_i - z_k \right |
\le   \max\limits_{1\le k \le n} \left |z_k \right |.
$$

The conjecture remains unsolved although attempts to verify this conjecture have
led to many interesting research results. The readers are referred
to the survey papers
\cite{Sch}, \cite{Sendov} as well as the two excellent books on the analytic
theory of polynomials, \cite{RS} and \cite{SS}.
Another conjecture relating the zeros and critical
points of a polynomial is the Schoenberg's conjecture \cite{Schoen}.
Let $z_1, z_2, \ldots , z_n$
be the zeros of a polynomial $p = c_nx^n + \ldots + c_0$ of degree $n$, 
$w_1, w_2, \ldots , w_{n-1}$ be the critical points of $p,$ and let 
$G = (1/n)\sum\limits_{i = 1}^{n} z_i $ be the   arithmetical mean of
the zeros of a polynomial $p$. It can be readily seen that this value
is equal to the arithmetical mean  of the critical points of
$p$, $G =   (1/(n-1))\sum\limits_{i = 1}^{n-1} w_i.$
Indeed, by Vieta's formulas applied for $p$ we obtain
$\sum\limits_{i = 1}^{n} z_i = -\frac{c_{n-1}}{c_n}$.
If we now apply Vieta's formulas for
$p' = nc_nx^{n-1} + (n-1)c_{n-1}x^{n-2} + \ldots + c_1$ we find that
$\frac{1}{n-1}\sum\limits_{i = 1}^{n-1} w_i =
-\frac{1}{n-1}\frac{(n-1)c_{n-1}}{n c_n} = -\frac{c_{n-1}}{nc_{n}} =
\frac{1}{n}\sum\limits_{i = 1}^{n} z_i.$
Till the end of this section $G=G(p)$ denotes this value.
In this notation the Schoenberg's conjecture
can be written as
$$\sum\limits^{n-1}_{i=1} |w_i|^2 \le
|G|^2 +
\frac{n-2}{n} \sum\limits^n_{i=1} |z_i|^2.$$

It is
natural to ask if one can bound $\sum\limits^n_{i=1} |z_i|^2$ by some
expressions in $w_i$ similar to those in Schoenberg's conjecture.
Our results on matrix integrability and full integrability of
polynomials (Theorem \ref{THM:integrable_iff_fullintegral}) are then
applied to prove
the   dual version of the Schoenberg inequality. Namely, this inequality
provides a bound for the sum of  squares of the absolute
values of zeros by an expression
in the critical points.

\begin{theorem}\label{THM:dual Schoenberg}
Following the Notation \ref{DEF:f_q_Q_h}, consider a degree $n$ polynomial
of type $(k, m)$ given by 
$$f(x) = (x - a_1)\ldots(x - a_k)
(x - b_1)^{\alpha_1}\ldots(x - b_m)^{\alpha_m},$$
where $a_1,\ldots, a_k, b_1, \ldots, b_m\in \C$ are pair-wise distinct,
$k, m \in \mathbb N \cup \{0\},\ \alpha_1,\ldots, \alpha_m \in
\mathbb N \setminus \{1\}.$ Let $h(x) := (x - a_1)\ldots(x - a_k)$ and
$h_{a_i}(x) := \frac{h(x)}{x - a_i}$.

Suppose $f$ has a full integral $F$ and $z_1,...,z_{n+1}$ are the zeros of $F$ 
and denote
$G = \frac{1}{n+1}\sum\limits_{i = 1}^{n+1} z_i =
\frac{1}{n}
\left(\sum\limits^{k}_{i=1} a_i + \sum\limits^{m}_{i=1} \alpha_ib_i\right)$
then
$$
\sum\limits^{n+1}_{i=1} |z_i|^2 \le \sum\limits^{k}_{i=1} |a_i|^2 +
\sum\limits^{m}_{i=1} \alpha_i|b_i|^2 + |G|^2 + 2(n+1)
\sum\limits^{k}_{i=1} \left|\frac{F(a_i)}{h_{a_i}(a_i)}\right|
$$
with equality holds if and only if
$\frac{F(a_i)}{h_{a_i}(a_i)}(\overline{a_i - \tau(\mathcal{B})}),
\ i = 1, \ldots, k$
are real.
\end{theorem}
\begin{proof} We shall make use of the Schur inequality \cite[p. 56]{RS} which
says that if $\lambda_i(A)$ are eigenvalues of a square matrix $A$ of
order $n$, then
$$\sum\limits^{n}_{i=1} |\lambda_i(A)|^2 \le || A ||_{F}^{2},$$
and the equality holds if and only if $A$ is normal.

From  Corollary \ref{COR:integrator_formula_with_smallest_norm} we know that
if $f$ has a full integral $F$, then $\mathcal{A} = \begin{pmatrix}
\mathcal{B} & u^{\top} \\
v & \tau(\mathcal{B})
\end{pmatrix}$
is an integral of $\mathcal{B}$ with $p_{\mathcal{A}} = (n+1)F$ possessing the
smallest Frobenius norm 
$||\mathcal{A}||_{F}^{2} = ||\mathcal{B}||_{F}^{2} +
|\tau(\mathcal{B})|^2 + 2\sum\limits_{i = 1}^{k}
\left|\frac{(n+1)F(a_i)}{h_{a_i}(a_i)}\right|$, where
$\tau(\mathcal{B}) := \frac{tr(\mathcal{B})}{n}=G$,\ 
$v_1 = \ldots = v_{C_m} = 0$,\ 
$v_{d_i} = \sqrt{\frac{(n + 1)F(a_i)}{h_{a_i}(a_i)}},\ i = 1,\ldots,k$
and $ u = v$. Since $(n+1)F$ is a characteristic polynomial of
$\mathcal{A}$,  by the Schur inequality
$$
\sum\limits^{n+1}_{i=1} |z_i|^2 = \sum\limits^{n+1}_{i=1}
|\lambda_i(\mathcal{A})|^2 \le
|| \mathcal{A} ||_{F}^{2} = \sum\limits^{k}_{i=1} |a_i|^2 +
\sum\limits^{m}_{i=1} \alpha_i|b_i|^2 + |G|^2 +
2(n+1) \sum\limits^{k}_{i=1} \left|\frac{F(a_i)}{h_{a_i}(a_i)}\right|.
$$

The equality in the Schur inequality holds if and only if
$\mathcal{A}$ is normal, i.e.
$\mathcal{A}\mathcal{A}^*=\mathcal{A}^*\mathcal{A}$.
Direct computations show 
$$\mathcal{A}^{*}\mathcal{A} = \begin{pmatrix}
\overline{\mathcal{B}} \mathcal{B} + \overline{v}^{\top}v &
\overline{\mathcal{B}} u^{\top} + \tau(\mathcal{B})\overline{v}^{\top}\\
\overline{u}\mathcal{B}+ \overline{\tau(\mathcal{B})} v &
\overline{u}u^{\top} + |\tau(\mathcal{B})|^2
\end{pmatrix},\ 
\mathcal{A}\mathcal{A}^{*} = \begin{pmatrix}
\mathcal{B}\overline{\mathcal{B}} + u^{\top}{\overline u} &
\mathcal{B}\overline{v}^{\top} + \overline{\tau(\mathcal{B})}u^{\top}\\
v\overline{\mathcal{B}}+ \tau(\mathcal{B})\overline{u} &
v\overline{v}^{\top} + |\tau(\mathcal{B})|^2
\end{pmatrix}.$$
Since $\mathcal{B}\overline{\mathcal{B}} = \overline{\mathcal{B}}\mathcal{B}$
then
$$
\mathcal{A}\mathcal{A}^*-\mathcal{A}^*\mathcal{A} =
\begin{pmatrix}
O &
(\mathcal{B} - \tau(B))\overline{v}^{\top} -
(\overline{\mathcal{B} - \tau(\mathcal{B}}))v^{\top}\\
v(\overline{\mathcal{B} - \tau(\mathcal{B})}) -
\overline{v}(\mathcal{B} - \tau(\mathcal{B})) & O 
\end{pmatrix}.$$
Thus $\mathcal{A}$ is normal if and only if this matrix is 0.
Since $\mathcal{B}$ is diagonal, it is equivalent to
$v_{d_i}\overline{(a_i - \tau(\mathcal{B}))} =
\overline{v_{d_i}}(a_i - \tau(\mathcal{B})),\ i = 1, \ldots, k$.
Substituting the values $v_{d_i}$ we 
equivalently obtain
$$
\frac{F(a_i)}{h_{a_i}(a_i)}(\overline{a_i - \tau(\mathcal{B})}) = 
\overline{\left(\frac{F(a_i)}{h_{a_i}(a_i)}\right)}(a_i - \tau(\mathcal{B})),\ 
i = 1, \ldots, k.
$$
\end{proof}

If all the critical points of a polynomial $p$ are distinct, then $p$ is
a full integral of $p'$. It then follows from the case 1 of
Theorem \ref{THM:diagonal_integrable_classifications} that we have
the following dual Schoenberg type inequality.
\begin{cor}\label{COR:dual Schoenberg}
Let $p$ be a polynomial of degree $n$ with the zeros
$z_1, \ldots , z_n$ and the
critical points $w_1, \ldots , w_{n-1}$.
Let $G=\frac{1}{n-1}\sum\limits^{n-1}_{i=1}w_i =
\frac{1}{n}\sum\limits^{n}_{i=1}z_i$.
If all the critical points of $p$ are distinct, then
$$
\sum\limits^{n}_{i=1} |z_i|^2 \le  |G|^2 +
\sum\limits^{n-1}_{i=1} |w_i|^2  +
2n \sum\limits^{n-1}_{i=1} \left|\frac{p(w_i)}{p''(w_i)}\right|
$$ 
with equality holds if and only if all elements
$\frac{p(w_i)}{p''(w_i)}\overline{(w_i - G)}, \ i = 1, \ldots, n - 1$ are real.
\end{cor}
\begin{proof}
Without loss of generality we assume that the coefficient at the
highest degree of $p'(x)$ is $1.$ Since $w_1,\ldots, w_{n-1}$ are distinct,
then in notations of previous
theorem $h(x) = p'(x), \  F(x) = p(x).$
Thus $p''(x) = \sum\limits_{i = 1}^{n - 1} \frac{p'(x)}{x - w_i}$ and
$p''(w_i) = \frac{p'}{x - w_i}(w_i).$
Therefore $\frac{F(w_i)}{h_{w_i}(w_i)} = \frac{p(w_i)}{p''(w_i)}$ and
we obtain the desired formula.
\end{proof}

\begin{rmk} Note that  the extra term
$2n \sum\limits^{n-1}_{i=1} \left|\frac{p(w_i)}{p''(w_i)}\right|$
is indeed necessary to
bound  $\sum\limits^{n}_{i=1} |z_i|^2$.  
To show this we consider the polynomial $p(z)=z^n-z, \ n > 1$.
Then $p'(z) = nz^{n-1} - 1, \  p''(z) = (n - 1)nz^{n-2}.$
Thus 
$\sum\limits^{n}_{i=1} |z_i|^2 = n - 1$ and
$\sum\limits^{n-1}_{i=1} w_i = \begin{cases}
\frac{1}{2}, \ \text{if } n = 2,\\
0, \ \text{if } n > 2.
\end{cases}
$ Moreover
$
\sum\limits^{n-1}_{i=1} |w_i|^2 = (n - 1)n^{\frac{2}{1 - n}}$.

\end{rmk}

\medskip
It is obvious that Sendov conjecture is equivalent to saying that all
the zeros $z_i$ of a polynomial of the degree $n$    lie in the union
$G = \bigcup\limits_{i=1}^{n-1} G_i$ of the disks with the center at the
critical point $w_i:$
\[
G_i = \{z\in\C: |z-w_{i}| \le 
\max\limits_{1\le k \le n} \left |z_k \right |\},\quad i=1,\ldots,n-1\ .
\]

If we consider polynomials as characteristic polynomials of certain matrices
it is  tempting to combine the Gerschgorin's theorem on the location of the
eigenvalues   together with the integration technique for matrices with
simple eigenvalues to study Sendov conjecture.  
To state Gerschgorin's theorem, for any square matrix $A=(a_{ij})$ of
order $n \ge 2$, we shall use the following notation:
\[
R_i(A) = \sum^n_{\genfrac{}{}{0pt}{}{j=1}{j\ne i}} |a_{ij}|,\quad i=1,\ldots,n.
\]
\begin{theorem} {\bf (Gerschgorin's theorem)} \rm{(\cite[p.344]{HJ})}. 
The eigenvalues of any square matrix $A=(a_{ij})$  of order $n \ge 2$
lie in the union $\bigcup\limits_{i=1}^n D_i$ of the Gerschgorin disks
\[
D_i = \{z\in\C: |z-a_{ii}| \le R_i(A)\}\ ,\quad i=1,\ldots,n.
\]
\end{theorem}

\medskip
Combining the Gerschgorin's theorem with the integration we obtain the
following theorem.

\begin{theorem}
Let $p$ be a polynomial of degree $n$ with zeros $z_1, z_2, \ldots, z_n$
and distinct critical points $w_1, w_2, \cdots , w_{n-1}$.
Then each zero is lies in the union $\bigcup\limits_{i=1}^n D_i$ of
$n$ disks
$D_i = \{z\in\C: |z-w_{i}| \le 
\max\limits_{1\le j \le n} \left |z_j \right |\}\,\quad i=1,\ldots,n-1$
and
$D_n=\{z\in\C: \left|z-\frac{1}{n}\sum\limits_{i = 1}^{n-1} w_i\right| \le
\frac{n}{\max\limits_{1\le j \le n} \left |z_j\right|}
\sum\limits^{n-1}_{i=1}\left|\frac{p(w_i)}{p''(w_i)}\right|\}$.
\end{theorem}

\begin{proof}
Consider $\mathcal{B} = {\rm diag\,}(w_1, \ldots, w_{n-1}).$ Since
$w_1, \ldots, w_{n-1}$ are distinct then $p$ is a full integral of
$p_{\mathcal{B}}(x).$ Consider an integral $\mathcal{A}$ of $\mathcal{B}$
given by formula from Corollary
\ref{COR:integrator_formula}.
For any $s \neq 0$, $\mathcal{A}$
is similar to  $\mathcal{A}_0=\begin{pmatrix}
\mathcal{B}& su^{\top} \\ 
s^{-1}v & \tau(\mathcal{B}) 
\end{pmatrix}$ because
$$
\begin{array}{lll}
& & 
\begin{pmatrix}
\mathcal{B}& su^{\top} \\ 
s^{-1}v & \tau(\mathcal{B}) 
\end{pmatrix} = \begin{pmatrix}
sI_n  & O \\ 
O & 1
\end{pmatrix} \ 

\begin{pmatrix}
\mathcal{B} & u^{\top} \\ v & \tau(\mathcal{B}) 
\end{pmatrix} \
\begin{pmatrix}
s^{-1}I_n & O \\ O & 1
\end{pmatrix} 
\end{array}
$$

Now for $s=\max\limits_{1\le j \le n} \left |z_j \right | > 0$ we have 
$R_i(\mathcal{A}_0)= s|u_i| = \max\limits_{1\le j \le n}
\left |z_j \right |,\ 1\le i \le n-1$
and
$R_n(\mathcal{A}_0)= s^{-1}\sum\limits_{i = 1}^{n-1} |v_i| =
\frac{1}{\max\limits_{1\le k \le n} \left |z_k \right |}
\sum\limits^{n-1}_{i=1}\left|\frac{np(w_i)}{p''(w_i)}\right|$.
Hence by the Gerschgorin's theorem we obtain that 
$$z_l \in \left(\bigcup\limits_{i=1}^{n-1} \{z\in\C: |z-w_{i}| \le 
\max\limits_{1\le j \le n} \left |z_j \right |\}\right) \bigcup
\{z\in\C: |z-\tau(\mathcal{B})| \le
\frac{n}{\max\limits_{1\le j \le n} \left |z_j\right|}
\sum\limits^{n-1}_{i=1}\left|\frac{p(w_i)}{p''(w_i)}\right|\}, 
$$ $l = 1,\ldots, n.$ 
Finally, the equality $\tau(\mathcal{B}) := \frac{1}{n}tr(\mathcal{B}) =
\frac{1}{n}\sum\limits_{i = 1}^{n-1} w_i$ yields the statement of the theorem. 
\end{proof}

\begin{rmk}
The size of the disk $D_n$ can be quite big so that all the zeros $z_i$
are lying inside it and in this case one cannot obtain information
about the relative position between the $z_i$ and $w_j$. 
\end{rmk}

\section*{Acknowledgments}

Investigations of integrability for diagonalizable matrices (Theorem~\ref{THM:diagonal_integrable_classifications}) are supported by the Ministry of Science and Higher Education
of the Russian Federation (Goszadaniye No. 075-00337-20-03, project No.
0714-2020-0005). Necessary and sufficient conditions for a matrix integral to be diagonalizable (Theorem \ref{THM:diagonal_integral_diagonalizable_criteria}) are obtained under the financial support of the Russian Federation Government (Grant number 075-15-2019-1926).

\end{document}